\documentclass{amsart}
\usepackage{color, latexsym, graphicx, amsmath, amssymb, cite, hyperref, tikz-cd}
\usepackage[margin=1.5in]{geometry}

\newtheorem{theorem}{Theorem}[section]
\newtheorem{lemma}[theorem]{Lemma}
\newtheorem{proposition}[theorem]{Proposition}
\newtheorem{corollary}[theorem]{Corollary}

\theoremstyle{definition}
\newtheorem{definition}[theorem]{Definition}
\newtheorem{example}[theorem]{Example}

\theoremstyle{remark}

\numberwithin{equation}{section}

\newcommand{\R}{{\mathbb R}}

\newcommand{\Z}{{\mathbb Z}}
\newcommand{\C}{{\mathbb C}}
\newcommand{\N}{{\mathbb N}}
\newcommand{\T}{{\mathbb T}}

\newcommand{\supp}{\operatorname{supp}}
\newcommand{\ord}{\operatorname{ord}}
\newcommand{\WF}{\operatorname{WF}}

\newcommand{\inj}{\operatorname{inj}}
\newcommand{\esssupp}{\operatorname{ess\, supp}}
\newcommand{\coker}{\operatorname{coker}}

\newcommand{\area}{\operatorname{area}}
\newcommand{\vol}{\operatorname{vol}}

\title{Triangles and triple products of Laplace eigenfunctions}
\author{Emmett L. Wyman}
\thanks{This work is supported in part by NSF Grant DMS-1502632.}
\begin{document}

\maketitle

%%%%%%%%%%%%
% ABSTRACT %
%%%%%%%%%%%%

\begin{abstract}
Consider an $L^2$-normalized Laplace-Beltrami eigenfunction $e_\lambda$ on a compact, boundary-less Riemannian manifold with $\Delta e_\lambda = -\lambda^2 e_\lambda$. We study eigenfunction triple products
\[
	\langle e_\lambda e_\mu, e_\nu \rangle = \int e_\lambda e_\mu \overline{e_\nu} \, dV.
\]
We show the overall $\ell^2$-concentration of these triple products is determined by the measure of some set of configurations of triangles with side lengths equal to the frequencies $\lambda,\mu,$ and $\nu$. A rapidly vanishing proportion of this mass lies in the `classically forbidden' regime where $\lambda, \mu,$ and $\nu$ fail to satisfy the triangle inequality. As a consequence, we refine one result in a paper by Lu, Sogge, and Steinerberger \cite{LuSoggeSteinerberger}.
\end{abstract}

%%%%%%%%%%%%%%%%
% Introduction %
%%%%%%%%%%%%%%%%

\section{Introduction} 

\subsection{The problem}

In what follows, $M$ will be a compact Riemannian manifold without boundary, and $e_1, e_2, \ldots$ will constitute an orthonormal basis of Laplace-Beltrami eigenfunctions with
\[
	\Delta e_j = -\lambda_j^2 e_j.
\]
We refer to $\lambda_j$, rather than $-\lambda_j^2$, as the eigenvalue or frequency of $e_j$.

We are interested in the general behavior of the product of two eigenfunctions. Rarely is the product of two eigenfunctions again an eigenfunction, but the product does admit a harmonic expansion
\begin{equation}\label{harmonic expansion}
	e_i e_j = \sum_k \langle e_i e_j, e_k \rangle e_k
\end{equation}
as a sum of eigenfunctions $e_k$. We are concerned with the following broad question: \emph{At what frequencies is the bulk of the $L^2$ mass of the product of two eigenfunctions located?} In other words, we want to determine for which frequencies $\lambda_k$ the coefficients $\langle e_i e_j, e_k \rangle$ must be large or small. %The bilinear forms $\langle e_j, e_k \rangle$ are not as interesting to study---their behavior is trivial by orthogonality.

Like many spectral problems, a form of this question was originally studied in number theory. There, the manifold is compact (or finite area cusped) and has constant sectional curvature $-1$. The problem was to show the coefficients $\langle e_j^2, e_k \rangle$ decay exponentially to order $\exp(-\pi \lambda_k/2)$, for fixed $e_j$. Such bounds are related to the Lindel\"of hypothesis for Rankin-Selberg zeta functions. The desired decay was obtained by Sarnak \cite{Sarnak-Integrals} up to a power loss, and optimal decay was later obtained by Bernstein and Reznikov \cite{Bernstein-Reznikov} and Kr\"otz and Stanton \cite{Krotz-Stanton}. Sarnak's result was extended to real-analytic manifolds by Zelditch \cite{Zpluri}. % Triple products are also closely related to Clebsch-Gordan coefficients.

This question has gained some recent interest in the (non-analytic) Riemannian setting \cite{Lu-Steinerberger, LuSoggeSteinerberger, Steinerberger}. Recent work is motivated in part by numerical applications, particularly to fast algorithms for electronic structure computing \cite{LuYing} which require most of the $L^2$ mass of $e_i e_j$ to be spanned by comparatively few basis eigenfunctions. This is quantified in work by Lu, Sogge, and Steinerberger \cite{Lu-Steinerberger,LuSoggeSteinerberger}, which among other things yields the following: For all $\epsilon > 0$ and integers $N \geq 1$, there exist constants $C_{\epsilon,N}$ for which
\begin{equation}\label{Lu Sogge Steinerberger}
	\sum_{\lambda_k \geq \lambda^{1 + \epsilon}} |\langle e_i e_j, e_k \rangle|^2 \leq C_{\epsilon, N} \lambda^{-N} \qquad \text{ for all $i$, $j$ for which $\lambda_i, \lambda_j \leq \lambda$.}
\end{equation}
This result shows only a rapidly-vanishing proportion of the $L^2$ mass of the product $e_i e_j$ lies at frequencies above $\lambda^{1 + \epsilon}$ if $\lambda_i, \lambda_j \leq \lambda$. In the present paper, we are able to replace $\lambda^{1 + \epsilon}$ under the sum with $(2 + \epsilon) \lambda$ and obtain the same bounds (see Corollary \ref{corollary}). Furthermore, the $(2 + \epsilon) \lambda$ factor is nearly optimal in the sense that there are examples for which the bound fails if the sum is instead over $\lambda_k \geq (2 - \epsilon) \lambda$ (see Example \ref{torus example}).

In a recent paper \cite{Steinerberger}, Steinerberger introduces a local correlation functional which quantifies the degree to which the oscillations of two eigenfunctions are parallel or transversal. He provides a beautiful identity between the local correlation functional and the heat evolution of the product $e_i e_j$, and uses this relationship to deduce information about where the spectral mass of $e_i e_j$ is concentrated. His results are quite general and also apply to graph Laplacians.

%Triple products tie in nicely to the theory of trigonometric polynomials on manifolds. This is described in Section \ref{TRIG POLYS}.

\subsection{Triple products and counting triangles}

Instead of taking the triple products $\langle e_i e_j, e_k \rangle$ individually, we study the sums
\begin{equation}\label{single sums}
	\sum_{(\lambda_i, \lambda_j, \lambda_k) = (\lambda, \mu, \nu)} |\langle e_i e_j, e_k \rangle|^2
\end{equation}
for given eigenvalues $\lambda, \mu$, and $\nu$. These sums are independent of the choice of eigenbasis and hence have behavior that is determined solely by the underlying manifold and its metric.

The main point of the present paper is to show the sum \eqref{single sums} can be estimated by counting configurations of triangles with side lengths $\lambda$, $\mu$, and $\nu$. This is best illustrated in the example on the torus.

\begin{example}\label{torus example}
Let $\T^n = \R^n/2\pi \Z^n$ denote the flat torus. The standard basis of eigenfunctions consists of Fourier exponentials of the form
\[
	(2\pi)^{-n/2} e^{i\langle x, m \rangle} \qquad \text{ for $m \in \Z^n$}
\]
with corresponding eigenvalue $|m|$. For $m, j, k \in \Z^n$, the triple product reads
\[
	\langle e_m e_j, e_k \rangle = (2\pi)^{-3n/2} \int_{\R^n/2\pi \Z^n} e^{i\langle x, m + j - k \rangle} \, dx =
	\begin{cases}
		(2\pi)^{-n/2} & \text{if } m + j - k = 0, \\
		0 & \text{otherwise.}
	\end{cases}
\]
The sum \eqref{single sums} is then
\[
	(2\pi)^{-n/2} \#\{ (m,j) \in \Z^{n + n} : |m| = \lambda, \ |j| = \mu, \ |m + j| = \nu \}.
\]
Each element in this set corresponds to a triangle with one vertex at the origin, the other two vertices at integer lattice points, and with side lengths $\lambda$, $\mu$, and $\nu$ (see Figure \ref{torusfigure}). If, for example, $\lambda, \mu,$ and $\nu$ fail to satisfy the triangle inequality, there are no triangles to count and \eqref{single sums} is zero. Obtaining optimal estimates for \eqref{single sums} in this example is highly nontrivial and closely tied to problems in geometric combinatorics and number theory (e.g. \cite{Ziegler}).
\end{example}

\begin{figure}
\includegraphics[width=0.5\textwidth]{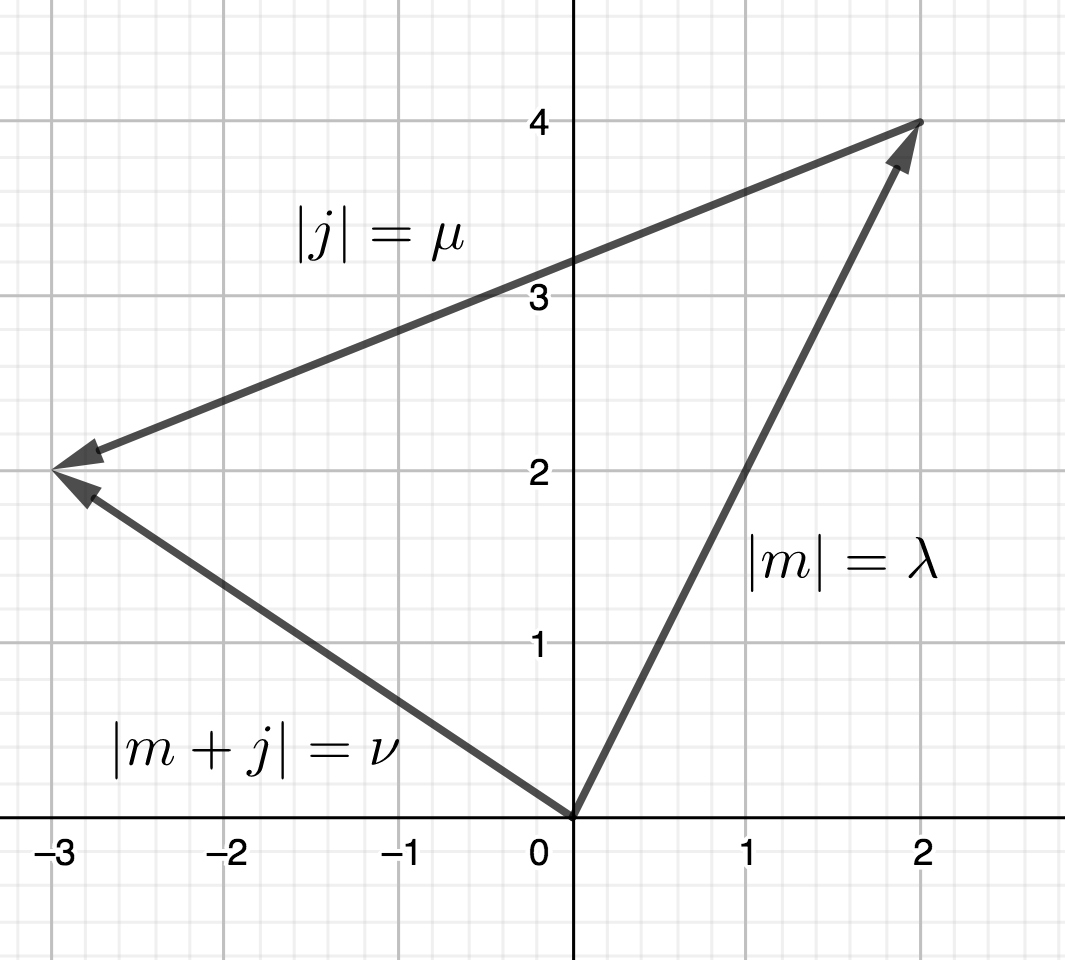}
\caption{An example of one of the triangular configurations in Example \ref{torus example}.}
\label{torusfigure}
\end{figure}

The natural next question is this: Is there a relationship between the sum \eqref{single sums} and the count of  configurations of triangles for manifolds in general? We answer this question by applying the theory of Fourier integral operators to an averaged version of \eqref{single sums}, where an estimation of the count of such configurations of triangles will arise in the asymptotics. Our argument also shows that \eqref{single sums} vanishes rapidly in the `classically forbidden' regime where $\lambda$, $\mu$, and $\nu$ fail to satisfy the triangle inequality.

%% LEFT OFF HERE
%\edit{LEFT OFF HERE}

\subsection{Statement of results} \label{STATEMENT OF RESULTS}

As before, $M$ denotes a compact Riemannian manifold without boundary, and $e_1, e_2, \ldots$ will constitute an orthonormal basis of Laplace-Beltrami eigenfunctions. Furthermore, we use $n$ to denote the dimension of $M$. Here $n$ will always be assumed to be $2$ or more, since the one-dimensional case is well-understood.

Our main theorems concern the general distribution of a joint spectral measure
\begin{equation}\label{joint spectral measure}
	\mu = \sum_{i,j,k} |\langle e_i e_j, e_k\rangle|^2 \delta_{(\lambda_i,\lambda_j,\lambda_k)}
\end{equation}
on $\R^3$ weighted by the norm-squares of the coefficients $\langle e_i e_j, e_k\rangle$. For the sake of organization, we introduce some terminology.

\begin{definition} \label{triangle-good and trangle-bad} Let $\Gamma$ be a closed cone in $\R^3 \setminus 0$. We say that $\Gamma$ is \emph{triangle-good} if for each $(a,b,c) \in \Gamma$, we have
\[
	a < b + c, \quad b < a + c, \quad \text{ and } \quad c < a + b.
\]
We say $\Gamma$ is \emph{triangle-bad} if for each $(a,b,c) \in \Gamma$, at least one of
\[
	a > b + c, \quad b > a + c, \quad \text{ or } \quad c > a + b
\]
holds.
\end{definition}

Triangle-good cones contain points whose coordinates are realizable as the side lengths of a triangle in the plane, and triangle-bad cones contain no such points. Note, a triangle-good cone necessarily lies in the positive octant of $\R^3 \setminus 0$. Furthermore, both definitions for triangle-good and triangle-bad cones exclude points $(a,b,c)$ specifying the side lengths of a degenerate triangle, i.e. when $a = b + c$ or some permutation thereof holds. These degenerate cases are geometrically problematic, so we cut them away.

Our first theorem shows that all but a rapidly-vanishing proportion of the mass of $\mu$ occurs at points $(\lambda_i,\lambda_j,\lambda_k)$ whose coordinates are realizable as the sidelengths of a triangle.

\begin{theorem}\label{main theorem triangle-bad}
	Let $\Gamma$ be a triangle-bad cone. Then for every integer $N \geq 0$, there exists a constant $C_{\Gamma,N}$ for which
	\[
		\mu(\tau + [0,1]^3) = \sum_{(\lambda_i,\lambda_j,\lambda_k) \in \tau + [0,1]^3} |\langle e_i e_j, e_k \rangle|^2 \leq C_{\Gamma,N} |\tau|^{-N} \qquad \text{ for $\tau \in \Gamma$},
	\]
	where $\tau + [0,1]^3$ denotes the translated unit cube $[\tau_1, \tau_1 + 1] \times [\tau_2, \tau_2 + 1] \times [\tau_3, \tau_3 + 1]$.
\end{theorem}

For the torus (Example \ref{torus example}), all terms in the sum are zero for $|\tau|$ large enough, and so Theorem \ref{main theorem triangle-bad} holds trivially. What is interesting is that this behavior for the torus persists for general manifolds up to the presence of a thin tail. In some cases, e.g. compact hyperbolic manifolds \cite{Sarnak-Integrals}, these tails are present.

As a corollary to Theorem \ref{main theorem triangle-bad}, we are able to obtain a refinement of the bound in \eqref{Lu Sogge Steinerberger}.

\begin{corollary}\label{corollary}
	For each $\epsilon > 0$ and integer $N > 0$, there exist constants $C_{\epsilon,N}$ for which
	\[
		\sum_{\lambda_k \geq (1 + \epsilon)(\tau_1 + \tau_2)} |\langle e_i e_j, e_k \rangle|^2 \leq C_{\epsilon,N} (\tau_1 + \tau_2)^{-N} \qquad \text{ for } \lambda_i \leq \tau_1 \text{ and } \lambda_j \leq \tau_2. 
	\]
\end{corollary}

The proof is very short, so we include it here.

\begin{proof}[Proof of Corollary \ref{corollary}] We apply Theorem \ref{main theorem triangle-bad} to obtain bounds
\[
	\sum_{(\lambda_i,\lambda_j,\lambda_k) \in \tau + [0,1]^3} |\langle e_i e_j, e_k \rangle|^2 \leq C_{\Gamma_\epsilon, N} |\tau|^{-N} \qquad \text{ for } \tau \in \Gamma_\epsilon
\]
for the triangle-bad cone
\[
	\Gamma_\epsilon = \{(\tau_1,\tau_2,\tau_3) : \tau_3 \geq (1 + \epsilon)(\tau_1 + \tau_2)\}.
\]
Set $N \geq 2$ and $T_m = (1 + \epsilon)(\tau_1 + \tau_2) + m$. The sum in the corollary is then bounded above by
\begin{align*}
	\sum_{m = 0}^\infty \sum_{\lambda_k \in [T_m,T_m+1]} |\langle e_i e_j, e_k \rangle|^2 &\leq \sum_{m = 0}^\infty C_{\Gamma_\epsilon, N} (\tau_1 + \tau_2 + m)^{-N}\\
	&\leq \frac{1}{N-1} C_{\Gamma_\epsilon,N} (\tau_1 + \tau_2 - 1)^{-N+1}.
\end{align*}
The corollary follows.
\end{proof}

Our second theorem describes the mass of $\mu$ which falls within triangle-good cones, and this is where we will see a count of configurations of triangles appear. Before stating the result, we recall the definition of the Leray density and some of its properties. Let $F : \R^n \to \R^d$ be a smooth function and let $y \in \R^d$ be such that $dF$ has full rank at each point in the preimage $F^{-1}(y)$. By the implicit function theorem, $F^{-1}(y)$ is a smooth $(n-d)$-dimensional manifold and admits a local parametrization by $z$ in a neighborhood of $\R^{n-d}$. Complete $z$ to a coordinate chart $(z,z')$ of a neighborhood in $\R^n$. The \emph{Leray density} on $F^{-1}(y)$ is given by
\begin{equation}\label{Leray density definition}
	d_F = \frac{|dx_1 \wedge \cdots \wedge dx_n|}{|dF_1 \wedge \cdots \wedge dF_d|}  = \left|\det \frac{\partial F}{\partial z'}\right|^{-1} \, dz,
\end{equation}
and is independent of the choice of the complementary coordinates $z'$. If $f$ is continuous on a neighborhood of $F^{-1}(y)$, we can identify $d_F$ with the distribution $\delta(F(x) - y)$ supported on $F^{-1}(y)$ by writing
\[
	\int_{F^{-1}(y)} f \, d_F = \int_{\R^n} f(x) \delta(F(x) - y) \, dx.
\]
%Provided $dF$ is full-rank on the support of a continuous function $f$ on $\R^n$, we have a convenient co-area formula,
%\[
%	\int_{\R^d} \left( \int_{F^{-1}(y)} f \, d_F \right) \, dy = \int_{\R^d} \int_{\R^n} f(x) \delta(F(x) - y) \, dx \, dy = \int_{\R^n} f(x) \, dx.
%\]
The \emph{Leray measure} of $F^{-1}(y)$ is defined as
\[
	\vol F^{-1}(y) = \int_{F^{-1}(y)} d_F.
\]
Note, the Leray measure need not coincide with the restriction of the Euclidean measure.
%The Leray measure admits an equivalent limit definition: Let $Q_\epsilon(y) = y + [\epsilon/2, \epsilon/2]^d$ be a cube centered at $y \in \R^d$. Then,
%\[
%	\vol F^{-1}(y) = \lim_{\epsilon \to 0} \frac{\vol F^{-1}(Q_\epsilon(y))}{\vol Q_\epsilon(y)}.
%\]
%The Leray density and Leray measure is similarly defined for maps $F : X \to Y$ between smooth manifolds $X$ and $Y$, provided $X$ and $Y$ are endowed with smooth densities, the one on $Y$ being nonvanishing.

Here and throughout the rest of the paper, we set $F : \R^{n+n} \to \R^3$ as
\begin{equation}\label{def F}
	 F(\xi,\eta) = (|\xi|,|\eta|,|\xi + \eta|).
\end{equation}
The level set $F^{-1}(\tau)$ is the set of pairs $(\xi,\eta)$ which specify a triangle of side lengths $\tau_1,\tau_2$ and $\tau_3$, much like in Example \ref{torus example}. Finally, we use the normalization
\[
	\hat f(\xi) = \int_{\R^n} e^{-i\langle x, \xi \rangle} f(x) \, dx \qquad \text{ and } \qquad f(x) = (2\pi)^{-n} \int_{\R^n} e^{i\langle x, \xi \rangle} \hat f(\xi) \, d\xi
\]
for the Fourier transform and Fourier inversion formula, and also use $\check f$ to denote the Fourier inversion of $f$. The next two statements provide a description of the concentration of $\mu$ in triangle-good cones.

\begin{theorem}\label{main theorem triangle-good}
Let $\rho$ be a Schwartz-class function on $\R^3$ with $\int \rho = 1$ and for which $\supp \hat \rho$ is contained in the open cube $(-\inj M, \inj M)^3$, and let $\Gamma$ be a triangle-good cone. Then for $\tau = (\tau_1, \tau_2, \tau_3) \in \Gamma$,
\[
	\rho * \mu(\tau) = (2\pi)^{-2n} \vol M \vol F^{-1}(\tau) + O(|\tau|^{2n-4}),
\]
where the constants implicit in the big-$O$ notation depend on $M$, $\Gamma$, and $\rho$ but not $\tau$.
\end{theorem}

\begin{proposition}\label{prop leray measure}
Let $\tau = (\tau_1, \tau_2, \tau_3)$ specify the side lengths of a nondegenerate triangle and let $\area(\tau)$ denote its area. Then,
\[
	\vol F^{-1}(\tau) = \vol S^{n-1} \vol S^{n-2} \tau_1 \tau_2 \tau_3 (2 \area(\tau))^{n-3}.
\]
\end{proposition}

Proposition \ref{prop leray measure} shows that the main term in Theorem \ref{main theorem triangle-good} is positive-homogeneous of order $2n-3$ and is indeed one order higher than the remainder. It is possible to obtain asymptotics for $\mu(\Omega)$ where $\Omega$ belongs to a suitable class of triangle-good regions using a Tauberian theorem like the one in \cite{WXZ}.

We have described how much of the mass of $\mu$ lies in both the triangle-good and triangle-bad regions of $\R^3$. What remains to be answered is how much of this mass can hide in the interface between, i.e. at points specifying the side lengths of degenerate triangles. The following corollary of the previous theorems shows only a vanishing proportion of the mass of $\mu$ may lie in this interface, though this can likely be improved. Recall our global assumption $n \geq 2$.

\begin{corollary}\label{main theorem triangle-degenerate}
Let $\chi$ be a Schwartz-class function on $\R$ with $\int \chi = 1$ and with Fourier support in $(-\inj M, \inj M)$ and let $\rho = \chi \otimes \chi \otimes \chi$. Then for $\tau_1,\tau_2 > 0$,
\begin{multline*}
	\int_{-\infty}^\infty \rho * \mu(\tau_1,\tau_2,\tau_3) \, d\tau_3 = (2\pi)^{-2n} (\vol M)  \int_{|\tau_1 - \tau_2|}^{\tau_1 + \tau_2} \vol F^{-1}(\tau_1,\tau_2,\tau_3) \, d\tau_3 \\ + O((\tau_1 + \tau_2)^{2n-3}).
\end{multline*}
\end{corollary}

The proof of is very short but requires pointwise asymptotics for eigenfunctions in addition to the asymptotics of Theorem \ref{main theorem triangle-good}.

\subsection*{Acknowledgements} The author would like to thank Alex Iosevich for helpful conversations and for his invaluable input on an early draft of this paper, and Steve Zelditch for helpful conversations and for his guidance and support as my postdoctoral mentor.

%%%%%%%%%
% Setup %
%%%%%%%%%

\section{Setup for Theorems \ref{main theorem triangle-bad} and \ref{main theorem triangle-good}}

We now prepare our problem for the theory of Fourier integral operators
%We will assume the reader is familiar with the basic theory---wavefront sets and their calculus, symbol classes and pseudodifferential operators, the symplectic structure of the cotangent bundle, half-densities (see also Appendix \ref{DENSITIES}), and the local definitions of Lagrangian distributions and Fourier integral operators and their symbols.
(see Duistermaat's book \cite{DuistermaatFIOs} and H\"ormander's paper \cite{HormanderPaper} for background). 
% See also books three and four of H\"ormander's series \cite{HormanderIII, HormanderIV} for a very thorough treatment.
In Section \ref{COMPOSITION FORMULA}, we use Duistermaat and Guillemin's \cite{DG} clean composition calculus to compute the symbol of the composition of a Fourier integral operator and a Lagrangian distribution. The author hopes that the presentation in Section \ref{COMPOSITION FORMULA} and Appendices \ref{DENSITIES} and \ref{SYMBOL CALCULUS} here will provide a helpful reference for those learning the clean composition calculus.

We begin by replacing $\langle e_i e_j, e_k \rangle$ by $\langle e_i e_j, \overline e_k \rangle$ in the definition of $\mu$ \eqref{joint spectral measure} and writing
\[
	\mu = \sum_{i,j,k} \left|\int_M e_i e_j e_k \, dV_M \right|^2 \delta_{(\lambda_i,\lambda_j,\lambda_k)}.
\]
This switch leaves $\mu$ unchanged and will reduce notation later. Note $\mu$ is tempered by standard sup-norm bounds on eigenfunctions (see e.g. \cite{HormanderIV, HangzhouLectures}).

In what follows, we use $P$ to denote the elliptic, first-order pseudodifferential operator $\sqrt{-\Delta}$ on $M$. To simplify notation, we set $m = (i,j,k) \in \N^3$ and take the smooth half-density
\[
	\phi_m = e_i \otimes e_j \otimes e_k |dV_{M^3}|^{1/2}
\]
on $M^3$ (see Appendix \ref{DENSITIES}). The collection $\phi_m$ for $m \in \N^3$ forms a Hilbert basis for the intrinsic $L^2$ space on $M^3$ consisting of joint eigen-half-densities of the commuting operators $P \otimes I \otimes I$, $I \otimes P \otimes I$, and $I \otimes I \otimes P$. We write its joint eigenvalues as $\lambda_m = (\lambda_i, \lambda_j, \lambda_k)$. Let $\Delta = \{(x,x,x) : x \in M\} \subset M^3$ denote the diagonal of the threefold product $M^3$, and let $\delta_\Delta$ to be the corresponding half-density distribution
\[
	(\delta_\Delta, f |dV_{M^3}|^{1/2}) = \int_M f(x,x,x) \, dV_M(x) \qquad \text{ for all $f \in C^\infty(M^3)$.}
\]
We then realize
\[
	\int_M e_i e_j e_k \, dV_M = (\delta_\Delta, \phi_m).
\]
The Fourier inversion of $\mu$ yields
\[
	\check \mu(t) = \frac{1}{(2\pi)^3} \sum_{m} |(\delta_\Delta, \phi_m)|^2 e^{i\langle t, \lambda_m \rangle}
\]
where here $t = (t_1,t_2,t_3) \in \R^3$. Now we write
\begin{align*}
	\check \mu(t) |dt|^{1/2} &= \frac{1}{(2\pi)^3} \sum_{m} (\delta_\Delta \otimes \delta_\Delta, \phi_m \otimes \overline \phi_m) e^{i\langle t, \lambda_m \rangle} |dt|^{1/2} \\
	&= \frac{1}{(2\pi)^3} U \circ \delta_{\Delta \times \Delta}(t)
\end{align*}
where here $U$ is the operator taking test half-densities on $M^3 \times M^3$ to half-density distributions on $\R^3$ with kernel
\[
	U(t,x,y) = e^{it_1P}(x_1,y_1) e^{it_2P}(x_2,y_2) e^{it_3 P}(x_3,y_3)
\]
for $x = (x_1,x_2,x_3) \in M^3$ and $y = (y_1,y_2,y_3) \in M^3$.

If we were to try to compute the symbolic data of the composition $U \circ \delta_{\Delta \times \Delta}$ right away, we quickly run into problems, namely that the composition is is not clean in the sense of Duistermaat and Guillemin \cite{DG}. We can circumvent this issue by applying a pseudodifferential operator $A \in \Psi^0_{\text{phg}}(M^3)$ to one of the factors of $\delta_\Delta$ to cut away problematic parts. We define $A$ on $M^3$ spectrally by
\begin{equation} \label{def A pdo}
	A\phi_m = a(\lambda_m) \phi_m,
\end{equation}
where $a$ is real, smooth, and positive-homogeneous of order $0$ outside of, say, the ball $B$ of radius $1$. By the conical support of $a$, we mean the smallest closed cone $\Gamma$ in $\R^3 \setminus 0$ for which $\supp a \setminus B \subset \Gamma$. We will need to make some decisions about the conical support of $a$ later in the argument, but for now we examine its effect on the calculation of $\check \mu(t) |dt|^{1/2}$. If we define an altered measure
\[
	\mu_a = \sum_m a(\lambda_m) |(\delta_\Delta, \phi_m)|^2 \delta_{\lambda_m},
\]
tracing back through the computations above yields
\begin{equation}\label{strategy main composition}
	\check \mu_a(t) |dt|^{1/2} = \frac{1}{(2\pi)^3} U \circ (\delta_\Delta \otimes A \delta_\Delta)(t).
\end{equation}
Note that if $\Omega$ is some Borel measurable subset of $\R^3$ on which $a \equiv 1$, then $\mu_a(\Omega) = \mu(\Omega)$.

The core of the argument is the computation of the symbolic data of the composition \eqref{strategy main composition}. Instead of performing the composition in one step, we break it into two. First, we consider the operator $Q : C^\infty(M^3) \to C^\infty(\R^3 \times M^3)$ with the same distribution kernel as $U$ and let $S : C^\infty(M^3) \to C^\infty(\R^3)$ be the operator with distribution kernel
\begin{equation}\label{strategy first composition}
	(Q \circ A) \circ \delta_\Delta.
\end{equation}
It follows that
\begin{equation}\label{strategy second composition}
	U \circ (\delta_\Delta \otimes A\delta_\Delta) = S \circ \delta_\Delta.
\end{equation}
The benefit of breaking the computation up is twofold: The composition \eqref{strategy first composition} is transversal and easy to compute, and the composition on the right hand side of \eqref{strategy second composition} involves fewer variables than $U \circ (\delta_\Delta \otimes A\delta_\Delta)$.

We now record the symbolic data of both $Q \circ A$ and $\delta_{\Delta}$ in preparation for the first composition. We begin with the latter. In what follows, $g$ will denote the Riemannian metric on $M$, and $|g(x)|^{1/2} \, dx$ the local form of the Riemannian volume density. For a general manifold $X$, we use $\dot T^*X$ to denote the punctured cotangent bundle $T^*X \setminus 0$. We use this `dot' notation to similarly denote punctured conormal bundles.

\begin{proposition}\label{delta symbol}
	$\delta_{\Delta} \in I^{n/4}(M^3, \dot N^*\Delta)$ where, with respect to canonical local coordinates,
	\[
		\dot N^*\Delta = \{(x,-\xi_2-\xi_3, x, \xi_2, x, \xi_3) : x \in M, \ \xi \in \dot T_{(x,x,x)}^*M^3 \}
	\]
	with a principal symbol with half-density part
	\[
		(2\pi)^{-n/4} \frac{1}{|g(x)|^{1/4}} |dx \, d\xi_2 \, d\xi_3|^{1/2}.
	\]
\end{proposition}

In the proof below, we use the definition of the principal symbol via oscillatory integrals with nondegenerate phase functions. To review, let $X$ be an open subset of $\R^n$ and let $\varphi(x,\theta)$ of $(x,\theta) \in X \times (\R^N \setminus 0)$ be a nondegenerate phase function, which requires $\varphi$ be positive-homogeneous in $\theta$, $d\varphi$ be nonvanishing, and $d\varphi_\theta'$ be full-rank wherever $\varphi_\theta' = 0$. We define a half-density distribution $u$ on $X$ by the oscillatory integral
\[
	u(x) = (2\pi)^{-(n + 2N)/4} \left( \int_{\R^N} e^{i\varphi(x,\theta)} a(x,\theta) \, d\theta \right) |dx|^{1/2},
\]
where $a(x,\theta)$ belongs to symbol class $S^m_{\text{phg}}(X \times (\R^N \setminus 0))$. The phase function defines a critical set
\[
	C_\varphi = \{(x,\theta) : \varphi_\theta'(x,\theta) = 0 \}.
\]
The critical set is a smooth conic manifold in $X \times (\R^N \times 0)$, and is diffeomorphic to a conic Lagrangian submanifold $\Lambda_\varphi \subset \dot T^*X$ via the map $(x,\theta) \mapsto (x,d_x \varphi(x,\theta))$. We say $u$ is a Lagrangian distribution associated to $\Lambda_\varphi$ of order
\[
	\ord u = m + (2N - n)/4,
\]
and write $u \in I^{m + (2N - n)/4}(X, \Lambda_\varphi)$. To $u$ we associate a homogeneous half-density on $\Lambda_\varphi$. This together with the Maslov index comprise the principal symbol of $u$. We only describe the half-density part, since that is the only part we need for our arguments. Given a parametrization of $C_\varphi$ by $\lambda$ in a subset of $\R^n$, the half-density part of the principal symbol is given by the transport of
\[
	a_0(\lambda) \sqrt{d_{\varphi_\theta'}} = a_0(\lambda) \left| \frac{\partial(\lambda, \varphi_\theta')}{\partial(x,\theta)} \right|^{-1/2} |d\lambda|^{1/2}
\]
to $\Lambda_\varphi$ via the map $C_\varphi \to \Lambda_\varphi$, where $a_0$ denotes the top-order term of $a$ and $d_{\varphi_\theta'}$ is the Leray density on $C_\varphi = (\varphi'_\theta)^{-1}(0)$.
%For more on the local definition of Lagrangian distributions and their principal symbols, see \edit{cite}.

\begin{proof}[Proof of Proposition \ref{delta symbol}] Let $x$ denote local coordinates of $M$, and $(x,y,z)$ the three-fold product of these coordinates. Note, $(x,y,z)$ parametrizes a neighborhood of $M^3$ intersecting the diagonal. Suppose $f$ is smooth and supported in such a coordinate patch of $M^3$. We have
\begin{multline*}
	(\delta_\Delta, f |dV_{M^3}|^{1/2}) = \int_{\R^n} f(x,x,x) |g(x)|^{1/2} \, dx \\
	= (2\pi)^{-2n} \idotsint e^{i(\langle x - y, \xi_2 \rangle + \langle x - z, \xi_3 \rangle)} f(x,y,z) |g(x)|^{1/2} \, dx \, dy \, dz \, d\xi_2 \, d\xi_3.
\end{multline*}
Recall, $f|dV_{M^3}|^{1/2}$ is written as $f(x,y,z) |g(x)|^{1/4}|g(y)|^{1/4}|g(z)|^{1/4} |dx \, dy \, dz|^{1/2}$ in local coordinates, and hence we write
\begin{align*}
	\delta_\Delta &= (2\pi)^{-2n} \left( \int_{\R^n} \int_{\R^n} e^{i(\langle x - y, \xi_2 \rangle + \langle x - z, \xi_3 \rangle)} \frac{|g(x)|^{1/4}}{|g(y)|^{1/4}|g(z)|^{1/4}} \, d\xi_2 \, d\xi_3 \right) |dx \, dy \, dz|^{1/2} \\
	&= (2\pi)^{-7n/4} \left( \int_{\R^n} \int_{\R^n} e^{i(\langle x - y, \xi_2 \rangle + \langle x - z, \xi_3 \rangle)} (2\pi)^{-n/4} \frac{|g(x)|^{1/4}}{|g(y)|^{1/4}|g(z)|^{1/4}} \, d\xi_2 \, d\xi_3 \right) |dx \, dy \, dz|^{1/2}.
\end{align*}
Here, we have pulled a factor of $(2\pi)^{-n/4}$ into the symbol so that the correct normalization appears in front.

We now check that $\phi(x,y,z,\xi_2,\xi_3) = \langle x - y, \xi_2 \rangle + \langle x - z, \xi_3 \rangle$ is a nondegenerate phase function. First, we note $\phi$ is positive homogeneous in the frequency variables $(\xi_2,\xi_3)$. Second, we have
\[
	\phi_{x,y,z}' =
		\begin{bmatrix}
		\xi_2 + \xi_3 \\
		-\xi_2 \\
		-\xi_3
		\end{bmatrix}
	\qquad \text{ and } \qquad
	\phi_\xi' =
		\begin{bmatrix}
		x - y \\
	 	x - z
		\end{bmatrix}
\]
and hence $d\phi \neq 0$ whenever $(\xi_2,\xi_3) \neq (0,0)$. Finally, $\phi$ has critical set
\[
	C_\phi = \{(x,x,x,\xi_2,\xi_3) : x \in \R^n, \ (\xi_2, \xi_3) \in \R^{2n} \setminus 0\},
\]
on which
\[
	d\phi_\xi' = \begin{bmatrix}
		I & -I & 0 & 0 & 0 \\
		I & 0 & -I & 0 & 0
	\end{bmatrix}
\]
has full rank. This means $\delta_\Delta$ is a Lagrangian distribution of order
\[
	\ord(\delta_\Delta) = -(3n - 4n)/4 = n/4.
\]
associated to the image of $\mathcal C_\phi$ through the map $(x,y,z,\xi_2,\xi_3) \mapsto (x,y,z,\phi_{x,y,z}'(x,y,z,\xi_2,\xi_3))$, namely
\[
	\{(x,x,x; \xi_2 + \xi_3, -\xi_2, -\xi_3) \in \dot T^* M^3 : x \in M, \ \xi_2, \xi_3 \in T_x^*M \},
\]
which is indeed the punctured conormal bundle $\dot N^*\Delta$. The half-density part of the symbol is the transport of the half-density
\[
	(2\pi)^{-n/4} |g(x)|^{-1/4} \sqrt{d_{\phi_{\xi}'}}
\]
via the parametrization of $\dot N^*\Delta$ by $x, \xi_2,$ and $\xi_3$, where $d_\phi$ is the Leray density
\[
	d_{\phi_{\xi}'} = \left|\frac{\partial(x,\xi_2,\xi_3, \phi_\xi')}{\partial(x,y,z,\xi_2,\xi_3)}\right|^{-1} \, dx \, d\xi_2 \, d\xi_3 = dx \, d\xi_2 \, d\xi_3.
\]
%The second equality in the line above comes from
%\[
%	\det \frac{\partial(x,\xi_2,\xi_3, \phi_\xi')}{\partial(x,y,z,\xi_2,\xi_3)}
%	=
%	\det \begin{bmatrix}
%		I & 0 & 0 & 0 & 0 \\
%		0 & 0 & 0 & I & 0 \\
%		0 & 0 & 0 & 0 & I \\
%		I & -I & 0 & 0 & 0 \\
%		I & 0 & -I & 0 & 0
%	\end{bmatrix}
%	= -1.
%\]
Hence, we have obtained the half-density symbol
\[
	(2\pi)^{-n/4} |g(x)|^{-1/4} |dx \, d\xi_2 \, d\xi_3|^{1/2}
\]
modulo a Maslov factor. The proposition follows after negating $\xi_2$ and $\xi_3$.
\end{proof}

Next, we describe the symbolic data of $Q \circ A$, for which we will need the Hamilton flow. In general, for a smooth manifold $X$ and a smooth function $p$ on $\dot T^*X$, the Hamilton vector field $H_p$ is defined locally by
\begin{equation}\label{local hamiltonian}
	H_p = \sum_j \left( \frac{\partial p}{\partial \xi_j} \frac{\partial}{\partial x_j} - \frac{\partial p}{\partial x_j} \frac{\partial}{\partial \xi_j} \right),
\end{equation}
where $(x_1,\ldots, x_n, \xi_1,\ldots, \xi_n)$ are canonical local coordinates of $T^*X$. $H_p$ can also be defined in a coordinate-invariant way as the vector field for which $\omega(H_p, v) = dp(v)$ for all vector fields $v$ on $\dot T^*X$, where $\omega = d\xi \wedge dx$ is the symplectic $2$-form on $T^*X$. The flow $\exp(tH_p)$ is called the Hamilton flow of $p$. Note, if $p$ is positive-homogeneous of order $1$, then its Hamilton flow is homogeneous in the sense that
\[
	\exp(tH_p)(x,\lambda \xi) = \lambda \exp(tH_p)(x,\xi) \qquad \text{ for $\lambda > 0$.}
\]
For further reading on Hamilton vector fields, flows, and the basics of symplectic geometry as it applies to Fourier integral operators, see \cite[\S21.1]{HormanderIII}, \cite[Chapter 3]{DuistermaatFIOs}, and \cite[\S4.1]{SoggeFIOs}.

Returning to our problem, we let $(x,\xi) = (x_1,\xi_1,x_2,\xi_2,x_3,\xi_3) \in T^*M^3$ for which $\xi_1,\xi_2,\xi_3$ are each nonzero and write
\[
	G^t(x,\xi) = ( \exp(t_1H_p) (x_1,\xi_1), \exp(t_2H_p) (x_2,\xi_2), \exp(t_3H_p) (x_3,\xi_3) ),
\]
where $\exp(t_i H_p)$ denotes the time-$t_i$ Hamiltonian flow with respect to the principal symbol $p$ of $P = \sqrt{-\Delta}$. By an abuse of notation, we also write
\[
	p(x,\xi) = (p(x_1,\xi_1), p(x_2,\xi_2), p(x_3,\xi_3)) \in \R^3,
\]
so that we may pretend $G^t$ is the Hamiltonian flow of an $\R^3$-valued symbol $p$ by a triple of times $t = (t_1,t_2,t_3)$.

\begin{proposition}\label{Q of A is an FIO} The following are true.
\begin{enumerate}
\item If the conical support of $a$ is triangle-good, then $Q \circ A$ of \eqref{strategy first composition} is a Fourier integral operator in $I^{-3/4}(\R^3 \times M^3 \times M^3, \mathcal C')$ associated to canonical relation
	\[
		\mathcal C = \{ (t, p(x,\xi), G^{-t}(x, \xi); x, \xi ) : t \in \R^3, \ (x,\xi) \in \esssupp A \}
	\]
	having principal symbol with half-density part
	\[
		(2\pi)^{3/4} a(p(x,\xi)) |dt \, dx \, d\xi|^{1/2}.
	\]
\item If the conical support of $a$ is triangle-bad, then $(Q \circ A) \circ \delta_\Delta$ is smooth.
\end{enumerate}
\end{proposition}

Before proceeding, we discuss how (2) implies Theorem \ref{main theorem triangle-bad}.

\begin{proof}[Proof of Theorem \ref{main theorem triangle-bad}] Let $a \equiv 1$ on $\Gamma$ and $a \equiv 0$ outside of a larger triangle-bad cone which contains $\Gamma$ in its interior. Let $\rho$ be a nonnegative Schwartz-class function on $\R^3$ with compact Fourier support and with $\rho(0) \geq 1$ on the cube $[0,1]^3$. The theorem is proved if $\rho * \mu_a(\tau) = O(|\tau|^{-\infty})$, i.e. if $\check \rho \check \mu_a$ is smooth. It suffices just to show $\check \mu_a$ is smooth, which follows if $S \circ \delta_\Delta$ is smooth, which follows if $S$ has a smooth kernel, which follows from (2) in the proposition. 
\end{proof}

We will need the the calculus of wavefront sets in the proof of Proposition \ref{Q of A is an FIO}.
%The wavefront set $\WF(u)$ of a distribution $u$ on a manifold $X$ is a closed conic subset of $\dot T^*X$ which contains point-momentum pairs $(x,\xi)$ for which $u$ localized to $x$ is singular in direction $\xi$. Given two manifolds $X$ and $Y$ and an operator $A : \mathcal D(Y) \to \mathcal D'(X)$ with distribution kernel $u \in \mathcal D'(X \times Y)$, the wavefront relation of $A$ is
%\[
%	\WF'(A) = \{(x,\xi; y,-\eta) \in T^*X \times T^*Y : (x,y,\xi,\eta) \in \WF(K) \}.
%\]
%In our context, one should think of conic Lagrangian submanifolds and canonical relations as more structured versions of wavefront sets and wavefront relations, respectively.
We refer the reader to Chapter 1 of Duistermaat's book \cite{DuistermaatFIOs} for a clear and detailed presentation.

\begin{proof}[Proof of Proposition \ref{Q of A is an FIO}]
Recall (e.g. from \cite[\S 29.1]{HormanderIV}) the symbolic data of the half-wave operator. We will briefly depart from our notation and take $t \in \R$, $x \in M$, and $\xi \in T_x^* M$. The half-wave operator $e^{-itP} : C^\infty(M) \to C^\infty(\R \times M)$ is a Fourier integral operator of order $-1/4$ with canonical relation
\[
	\{(t,-p(x,\xi), \exp(tH_p)(x,\xi), x, \xi) : t \in \R, \ (x,\xi) \in \dot T^*M \}
\]
with a principal symbol with half-density part
\[
	(2\pi)^{1/4} |dt \, dx \, d\xi|^{1/2}.
\]
By a change of variables $t \mapsto -t$, $e^{itP}$ is similar except that its canonical relation is instead
\[
	\{(t, p(x,\xi), \exp(-tH_p)(x,\xi); x, \xi) : t \in \R, \ (x,\xi) \in \dot T^*M \}.
\]

Let $u$ denote the distribution kernel of $e^{itP}$ on $\R \times M \times M$. Now we switch our notation back to $t \in \R^3$ and $(x,\xi) \in T^*M^3$. The kernel of $Q$ is, up to a permutation of the variables, the threefold tensor product $u \otimes u \otimes u$. Hence by \cite[Proposition 1.3.5]{DuistermaatFIOs}, the wavefront set of the kernel of $Q$ is contained in the union of
\begin{equation}\label{U lagrangian submanifold}
	\WF(u) \times \WF(u) \times \WF(u) \simeq \{(t, p(x,\xi), G^{-t}(x,\xi), x, -\xi) : t \in \R^3, \ (x,\xi) \in \dot T^*M^3 \}
\end{equation}
along with the three `planes'
\begin{align*}
	&0 \times \WF(u) \times \WF(u), \\
	&\WF(u) \times 0 \times \WF(u), \text{ and} \\
	&\WF(u) \times \WF(u) \times 0
\end{align*}
and the three `axes'
\begin{align*}
	&\WF(u) \times 0 \times 0, \\
	&0 \times \WF(u) \times 0, \text{ and }\\
	&0 \times 0 \times \WF(u),
\end{align*}
where the $0$'s occurring in the products denote the zero section of $T^*(\R \times M \times M)$.

If not for the planes and axes, the kernel of $Q$ would be a genuine Langrangian distribution associated with \eqref{U lagrangian submanifold}. To prove (1), it suffices to show $A$ cuts away these bad planes and axes. Since $p(x,\xi)$ lies in the positive octant for any $(x,\xi) \in \esssupp A$, each of $\xi_1, \xi_2$, and $\xi_3$ is nonzero, hence $(x,\xi)$ must not belong to any of the problematic planes or axes. We conclude
\[
	\WF'(Q \circ B) \subset \{(t, p(x,\xi), G^{-t}(x,\xi); x, \xi) : t \in \R^3, \ (x,\xi) \in \esssupp A \}.
\]
The principal symbol on this canonical relation is then the product of the symbols on the half-wave factors and the principal symbol of $A$, namely
\[
	(2\pi)^{3/4} a(p(x,\xi)) |dt \, dx \, d\xi|^{1/2}
\]
up to a Maslov factor.

To prove (2), we employ the composition calculus of wavefront sets \cite[Corollary 1.3.8]{DuistermaatFIOs}. Note $A \delta_\Delta$ is smooth, and so
\[
	\WF(Q \circ A \delta_\Delta) \subset \{(t,\tau,y,\eta) : (t,\tau,y,\eta,x,0) \in \WF(Q) \text{ for some $x \in M^3$}\}.
\]
So, suppose $(t,\tau,y,\eta,x,0) \in \WF(Q)$. Clearly, this point cannot lie in the main component \eqref{U lagrangian submanifold}, so it must lie in one of the troublesome planes or axes. However, none of the factors $\WF(u)$ contain elements of the form $(t_1,\tau_1,y_1,\eta_1,\xi_1,0)$ either, since then both $\eta_1 = 0$ and $\tau_1 = 0$ as well. Hence, the wavefront set of $Q \circ A \delta_\Delta$ is empty and we have (2).
\end{proof}

%Next, we compute the symbolic data of the first composition \eqref{strategy first composition}.

We now state the symbolic data of the composition $(Q \circ A) \circ \delta_\Delta$ in the case where the conical support of $a$ is triangle-good. We defer the proof until Section \ref{SYMBOLS} after we have built some helpful formulas in Section \ref{COMPOSITION FORMULA}.

\begin{proposition} \label{prop first composition}
	Suppose the conical support of $a$ is triangle-good. Then, $(Q \circ A) \circ \delta_\Delta$ belongs to class $I^{(n-3)/4}(\R^3 \times M^3; \Lambda)$ where
	\[
		\Lambda = \mathcal C \circ \dot N^*\Delta = \{ (t,p(x,\xi), G^{-t}(x,\xi))  \in \dot T^*(\R^3 \times M^3) : 
		 (x,\xi) \in \dot T^* M^3 \cap \esssupp A \}
	\]
	with principal symbol with half-density part
	\[
		(2\pi)^{-(n-3)/4} a(p(x,\xi)) |g(x_1)|^{-1/4} |dt \, dx_1 \, d\xi_2 \, d\xi_3|^{1/2},
	\]
	via the parametrization $(x,\xi) = (x_1,x_1,x_1,-\xi_2-\xi_3,\xi_2,\xi_3)$.
\end{proposition}

Recall from \eqref{strategy first composition} and \eqref{strategy second composition} that $S : C^\infty(M^3) \to C^\infty(\R^3)$ is the operator with distribution kernel $(Q \circ A) \circ \delta_\Delta$. That is, $S \in I^{(n-3)/4}(\R^3 \times M^3; \mathcal C')$ where
\begin{multline} \label{S canonical relation}
	\mathcal C = \{(t,p(x,\xi); G^{t}(x,\xi)) \in \dot T^*R^3 \times \dot T^*M^3: \\
	x = (x_1,x_1,x_1) \in M^3, \ \xi = (-\xi_2-\xi_3,\xi_2,\xi_3) \in \dot T^*_x M^3, \\
	(x,\xi) \in \esssupp A \}
\end{multline}
with the same symbol as in the proposition. Note, $\mathcal C$ is indeed a subset of $\dot T^*R^3 \times \dot T^*M^3$ provided $\esssupp A$ is triangle-good or triangle-bad. The wavefront set of $S \circ \delta_\Delta$ can now be computed with the standard calculus \cite[Chapter 1]{DuistermaatFIOs}, and in doing so the following configurations will appear.

\begin{definition}\label{def geodesic triple} Let $x_1 \in M$ and $\xi_1, \xi_2, \xi_3$ be three covectors over $x_1$, none of which are zero, for which $\xi_1 + \xi_2 + \xi_3 = 0$. Let $t_1, t_2, t_3$ be three times for which the $H_p$ flow of $(x_1,\xi_1), (x_1,\xi_2),$ and $(x_1,\xi_3)$ by times $t_1,t_2,$ and $t_3$, respectively, all lie over a common point in $M$ and
\[
	\exp(t_1 H_p)(x_1,\xi_1) + \exp(t_2 H_p)(x_1,\xi_2) + \exp(t_3 H_p)(x_1,\xi_3) = 0.
\]
We call such a configuration a \emph{geodesic triple} with data $(t,x_1,\xi_2,\xi_3)$.
\end{definition}

Given $t \in \R^3$, $x = (x_1,x_1,x_1) \in M^3$, and $\xi = (-\xi_2-\xi_3,\xi_2,\xi_3) \in \dot N^*_x\Delta$, we will allow ourselves to write the data $(t,x_1,\xi_2,\xi_3)$ of a geodesic triple as $(t,x,\xi)$. This will be convenient for expressing the wavefront set of the composition $S \circ \delta_\Delta$, summarized below.

\begin{proposition}\label{wavefront set prop}
	Suppose the conical support of $a$ is triangle-good. The composition $S \circ \delta_{\Delta}$ is a tempered distribution on $\R^3$ with
	\begin{multline*}
		\WF(S \circ \delta_{\Delta}) \subset \{(t,p(x,\xi)) \in \dot T^*\R^3 : \\
		(t,x,\xi) \text{ is the data of a geodesic triple and } (x,\xi) \in \esssupp A \}.
	\end{multline*}
\end{proposition}

\begin{proof}
	We are already aware that $S \circ \delta_\Delta$ is a tempered-distribution since it is the inverse Fourier transform of a tempered measure $\mu_a$ on $\R^3$. By \cite[Corollary 1.3.8]{DuistermaatFIOs}, we have
	\begin{align*}
		\WF(S \circ \delta_{\Delta}) &\subset \mathcal C \circ \dot N^*\Delta \\
		&= \{(t,p(x,\xi)) : (x,\xi) \in \esssupp A \cap \dot N^*\Delta \text{ such that } G^t(x,\xi) \in \dot N^*\Delta \}
	\end{align*}
	where $\mathcal C$ is the canonical relation of $S$ from \eqref{S canonical relation}. The proposition follows from the definition of $G^t$ and Definition \ref{def geodesic triple}.
\end{proof} 

Finally, we state the proposition which will comprise the bulk of our argument for Theorem \ref{main theorem triangle-good}.

\begin{proposition}\label{main composition prop} Suppose the conical support of $a$ is triangle-good. The restriction of the composition $S \circ \delta_{\Delta}$ to the open cube $(-\inj M, \inj M)^3$ is a Lagrangian distribution in class $I^{2n - \frac94}((-\inj M, \inj M)^3 ; \dot T_0^* \R^3)$ with principal symbol having half-density part
\[
	(2\pi)^{-2n + \frac94} \vol M a(\tau) \vol F^{-1}(\tau) |d\tau|^{1/2}.
\]
\end{proposition}

The proof of this proposition will be the subject of the next two sections and involve the clean composition calculus of Duistermaat and Guillemin \cite{DG}. For now, we show how the proposition implies Theorem \ref{main theorem triangle-good}.

\begin{proof}[Proof of Theorem \ref{main theorem triangle-good}] To distill the asymptotics of Theorem \ref{main theorem triangle-good} from Proposition \ref{main composition prop}, we test $S \circ \delta_\Delta$ against the oscillating half-density
\[
	\phi_\tau(t) = \check \rho(t) e^{-i\langle t, \tau \rangle} |dt|^{1/2},
\]
where $\rho$ is as in the theorem, with $\int \rho(\tau) \, d\tau = 1$ and $\supp \check \rho \subset (-\inj M, \inj M)^3$. On one hand, \eqref{strategy main composition} affords
\[
	(S \circ \delta_\Delta, \phi_\tau) = (2\pi)^3 \int_{\R^3} e^{-i\langle t, \tau \rangle} \check \mu_a(t) \check \rho(t) \, dt = \rho * \mu_a(\tau).
\]
On the other hand, \cite[(25.1.13)]{HormanderIV} and its preceding discussion yields
\[
	\check\rho S \circ \delta_\Delta = (2\pi)^{-9/4} \left(\int_{\R^3} e^{i\langle t,\tau \rangle} \nu(\tau) \, d\tau \right) |dt|^{1/2}
\]
where by Proposition \ref{main composition prop}
\[
	\nu(\tau) \equiv \zeta \check \rho(0) (2\pi)^{-2n + 9/4} \vol M a(\tau) \vol F^{-1}(\tau) \mod S^{2n-4}(\R^3 \setminus 0),
\]
where $S^{2n-4}(\R^3 \setminus 0)$ denotes the symbol class of order $2n-4$ and $\zeta$ is a complex unit coming from the neglected Maslov factors. By Fourier inversion,
\[
	(S \circ \delta_\Delta, \phi_\tau) = (2\pi)^{3-9/4} \nu(\tau) = \zeta \check \rho(0) (2\pi)^{3-2n} \vol M a(\tau) \vol F^{-1}(\tau) + O(|\tau|^{2n-4}).
\]
Next, we argue $\zeta = 1$. Select $a \geq 0$ and $\rho$ nonnegative. Then, $\rho * \mu_a(\tau)$ is real and positive and hence so must be $\zeta$. We also have $\check \rho(0) = (2\pi)^{-3}$ since $\int \rho = 1$. Putting everything together, we have
\[
	\rho * \mu_a(\tau) =(2\pi)^{3-9/4} \nu(\tau) = (2\pi)^{-2n} \vol M a(\tau) \vol F^{-1}(\tau) + O(|\tau|^{2n-4}).
\]
The theorem follows after selecting $a \equiv 1$ on the triangle-good cone $\Gamma$ and $a \equiv 0$ outside of a triangle-good cone containing $\Gamma$ in its interior.
\end{proof}

%%%%%%%%%%%%%%%%%%%%%%%
% Composition Formula %
%%%%%%%%%%%%%%%%%%%%%%%

\section{Two Convenient Composition Formulas}\label{COMPOSITION FORMULA}

We now state and prove two convenient composition formulas in the service of Propositions \ref{prop first composition} and \ref{main composition prop}. This will require Duistermaat and Guillemin's \cite{DG} symbol calculus of Fourier integral operators with cleanly composing canonical relations (see also \cite{HormanderIV}). This and the accompanying half-density formalism is reviewed in Appendices \ref{DENSITIES} and \ref{SYMBOL CALCULUS}.

\subsection{A composition formula for Proposition \ref{prop first composition}} \label{Composition Formula 1}

Let $Y$ be a smooth, compact manifold without boundary with $\dim Y = n$. Fix a conic Lagrangian submanifold $\Lambda \subset T^*Y$. Let $p = (p_1,\ldots,p_k)$ be a list of positive-homogeneous first-order symbols with Hamilton bracket $\{p_i, p_j \} = 0$ for all $i,j = 1,\ldots,k$. Note, each pair of vectors $H_{p_i}$ and $H_{p_j}$ commute since $\{p_i, p_j \} = 0$.

Consider the canonical relation
\[
	\mathcal C = \{(t,p(x,\xi), G^{-t}(x,\xi) ; x,\xi) : t \in \R^k, \ (x,\xi) \in \dot T^*Y \} \subset \dot T^*(\R^k \times Y) \times \dot T^*Y
\]
where
\[
	G^t = \exp(t_1 H_{p_1}) \circ \cdots \circ \exp(t_k H_{p_k}) = \exp(t_1 H_{p_1} + \cdots + t_k H_{p_k}).
\]
In what follows, we consider a parametrization $\lambda \mapsto (x(\lambda), \xi(\lambda))$ of $\Lambda$ by $\lambda = (\lambda_1,\ldots,\lambda_n) \in \R^n$. The composition
\[
	\mathcal C \circ \Lambda = \{(t,p(x, \xi), G^{-t}(x,\xi)) : t \in \R^k, \ (x,\xi) \in \Lambda \}
\]
then admits a parametrization by $t$ and $\lambda$.

\begin{lemma}\label{symbol composition lem trans}
	The composition of $\mathcal C$ and $\Lambda$ is transversal. Furthermore, fix homogeneous half-densities
	\[
		\sigma_{\mathcal C}(t,x,\xi) = a(t,x,\xi) |dt \, dx \, d\xi|^{1/2} \qquad \text{ and } \qquad \sigma_{\Lambda}(\lambda) = b(\lambda) |d\lambda|^{1/2}
	\]
	on $\mathcal C$ and $\Lambda$, respectively. Then the composition $\sigma_{\mathcal C} \circ \sigma_{\Lambda}$ is the half-density
	\[
		 a(t,x(\lambda),\xi(\lambda)) b(\lambda) \, |dt \, d\lambda|^{1/2}
	\]
	via the parametrization of $\mathcal C \circ \Lambda$ by $t$ and $\lambda$.
\end{lemma}

We outline the symbol calculus for the transversally composing case in Appendix \ref{SYMBOL CALCULUS} as a special case of the clean calculus. For a more direct approach, see \cite{DuistermaatFIOs, HormanderPaper}

\begin{proof}
	The fiber product of $\mathcal C$ and $\Lambda$ is
	\[
		F = \{(t,p(x,\xi), G^{-t}(x,\xi); x,\xi) : t \in \R^k, (x,\xi) \in \Lambda\}
	\]
	and can be parametrized by $t$ and $\lambda$ by $(x,\xi) = (x(\lambda),\xi(\lambda))$. One quickly checks that $F$ is the transverse intersection of $\mathcal C$ and $T^*(\R^k \times Y) \times \Lambda$ and hence $\mathcal C$ and $\Lambda$ compose transversally as per Appendix \ref{SYMBOL CALCULUS}.
	
	Next, we compute the half-density part of the symbol. Fix $(a;b) \in F$, where
	\[
		a = (t,p(x,\xi),G^{-t}(x,\xi)) \qquad \text{ and } \qquad b = (x,\xi).
	\]
	It suffices to identify the half-density on $T_{(a;b)} F$ induced by the short exact sequence
	\[
		0 \longrightarrow T_{(a;b)} F \longrightarrow T_{(a;b)}\mathcal C \times T_b \Lambda \overset \beta \longrightarrow T_bT^*Y \longrightarrow 0
	\]
	where $\beta((a';b'),c') = b' - c'$. Then, since $F \to \mathcal C \circ \Lambda$ is a diffeomorphism, $\sigma_{\mathcal C} \circ \sigma_\Lambda$ will be precisely our half density on $F$ up to a reparametrization.
	
	To proceed, we set some bases. Let $\mathbf e = (e_1,\ldots,e_n)$ denote the basis for $T_b \Lambda$ given by the coordinate vector frame
	\[
		e_j = \frac{\partial}{\partial \lambda_j}.
	\]
	Since $\Lambda$ is Lagrangian, we complete $\mathbf e$ to a symplectic basis $(\mathbf e, \mathbf f) = (e_1,\ldots,e_n, f_1, \ldots, f_n)$ of $T_b T^*Y$. Next, let $\mathbf g = (g_1,\ldots,g_k)$ denote the standard basis for $T_t \R^k$. Now, $T_{(a;b)} F$ has basis
	\begin{align*}
		&(g_j,0,*;0) &&j = 1,\ldots,k, \\
		&(0,*,*;e_j) &&j = 1,\ldots,n.
	\end{align*}
	Here, the asterisks denote vectors which are determined by others in the tuple. This basis is mapped to a linearly independent list in $T_{(a;b)} \mathcal C \times T_b \Lambda$, specifically
	\begin{align*}
		&(g_j,0,*;0) \times 0 &&j = 1,\ldots,k, \\
		&(0,*,*;e_j) \times e_j &&j = 1,\ldots,n.
	\end{align*}
	We complete this to a basis of $T_{(a;b)} \mathcal C \times T_b \Lambda$ by adding the vectors
	\begin{align*}
		&(0,*,*;e_j) \times 0 &&j = 1,\ldots,n, \\
		&(0,*,*;f_j) \times 0 &&j = 1,\ldots,n.
	\end{align*}
	After a determinant-$\pm 1$ transformation, we see the evaluation of $\sigma_{\mathcal C} \otimes \sigma_{\Lambda}$ on this basis is
	\[
		a(t,x(\lambda),\xi(\lambda)) b(\lambda).
	\]
	The extension maps forward through $\beta$ to a symplectic basis for $Y$, on which the symplectic half-density simply evaluates to $1$. We conclude that
	\[
		\sigma_{\mathcal C} \circ \sigma_\Lambda(\mathbf g, \mathbf e) = a(t, x(\lambda), \xi(\lambda)) b(\lambda),
	\]
	and hence
	\[
		\sigma_{\mathcal C} \circ \sigma_\Lambda = a(t, x(\lambda), \xi(\lambda)) b(\lambda) |dt \, d\lambda|^{1/2}
	\]
	as desired.
\end{proof}

\subsection{A composition formula for Proposition \ref{main composition prop}}

%Let $Y$ be smooth, compact manifold without boundary with $\dim Y = n$. Fix a conic Lagrangian submanifold $\Lambda \subset T^*Y$. Let $p = (p_1,\ldots,p_k)$ be a list of positive-homogeneous first-order symbols with Hamilton bracket $\{p_i, p_j \} = 0$ for all $i,j = 1,\ldots,k$, and where the list of respective Hamilton vector fields
%\[
%	H_p = (H_{p_1}, \ldots, H_{p_k})
%\]
%is linearly independent and whose span has trivial intersection with the fibers of $T\Lambda$. Note, each pair of vectors $H_{p_i}$ and $H_{p_j}$ commute since $\{p_i, p_j \} = 0$.

Let $Y$, $\Lambda$, and $p$ be as in the previous section, but additionally assume the list of Hamilton vector fields
\[
	H_p = (H_{p_1}, \ldots, H_{p_k})
\]
is linearly independent and have span which intersects the fibers of $T\Lambda$ trivially. We will also instead consider the canonical relation
\[
	\mathcal C = \{(t,p(x,\xi); G^t(x,\xi)) : t \in \R^k, \ (x,\xi) \in \Lambda \} \subset \dot T^*\R^k \times \dot T^*Y.
\]

\begin{lemma}\label{symbol composition lem}
	Let $\mathcal C$ and $\Lambda$ be as above. The following are true.
	\begin{enumerate}
	\item There is an isolated component of
	\[
		\mathcal C \circ \Lambda = \{(t,p(x,\xi)) \in \dot T^* \R^k : (x,\xi) \in \Lambda, \ G^t(x,\xi) \in \Lambda \}
	\]
	contained in $\dot T_0^*\R^k$.
	\item At this component, $\mathcal C$ and $\Lambda$ compose cleanly with excess
	\[
		e = \dim Y - k.
	\]
	\item Suppose $\Lambda$ is parametrized by $\lambda = (\lambda_1,\ldots,\lambda_n) \in \R^n$, and that we have homogeneous half-densities
	\[
		\sigma_{\mathcal C} = b(t,\lambda) |dt \, d\lambda|^{1/2} \qquad \text{ and } \qquad \sigma_{\Lambda} = a(\lambda) |d\lambda|^{1/2}
	\]
	on $\mathcal C$ and $\Lambda$, respectively. Then we have
	\[
		\sigma_{\mathcal C} \circ \sigma_\Lambda = \left(\int_{p^{-1}(\tau)} b(0,\lambda) a(\lambda) \, d_p(\lambda) \right) |d\tau|^{1/2}
	\]
	at $(0,\tau) \in \dot T^*\R^k$, where $d_p$ denotes the Leray density on $p^{-1}(\tau)$.
	\end{enumerate}
\end{lemma}

We split the proof into two parts, one deals with parts (1) and (2) and the other deals with part (3) and requires an additional lemma. 

\begin{proof}[Proofs of Lemma \ref{symbol composition lem} (1) and (2)] The fiber product of $\mathcal C$ and $\lambda$ is
\[
	F = \{(t,x,\xi) : (x,\xi) \in \Lambda, \ G^t(x,\xi) \in \Lambda\}
\]
with maps
\begin{align*}
	F \to \mathcal C &: (t,x,\xi) \mapsto (t,p(x,\xi) ; G^t(x,\xi)) \qquad \text{ and } \\
	F \to \Lambda &: (t,x,\xi) \mapsto (x,\xi).
\end{align*}
Let $f : T^*Y \to \R^n$ be a defining function of $\Lambda$, meaning both $f = 0$ and $df$ has full-rank on $\Lambda$. Then, we may rewrite $F$ as
\[
	F = \{(t,x,\xi) : (x,\xi) \in \Lambda, \ f(G^t(x,\xi)) = 0\}.
\]
First, since the vectors in $H_p$ are a basis for a space with trivial intersection with $T_{(x,\xi)} \Lambda = \ker df$, $df(H_p)$ is a linearly independent list in $\R^n$. In particular for fixed $(x,\xi) \in \Lambda$, $t \mapsto f(G^t(x,\xi))$ has injective differential, which implies (1).

To verify (2), we require $T_{(0,x,\xi)} F$ to be the fiber product of $T_{(0,p(x,\xi))} \mathcal C$ and $T_{(x,\xi)} \Lambda$. First, we determine
\[
	T_{(0,p(x,\xi); x,\xi)} \mathcal C = \{(t',dp(x',\xi') ; (x',\xi') + t'H_p) : t' \in T_0 \R^k , \ (x',\xi') \in T_{(x,\xi)} \Lambda \}
\]
where $t'H_p = t_1' H_{p_1} + \cdots + t_k' H_{p_k}$. The fiber product of the tangent spaces is
\[
	F' = \{(t',x',\xi') : t' \in T_0\R^k , \ (x',\xi') \in T_{(x,\xi)} \Lambda, \ (x',\xi') + t'H_p \in T_{(x,\xi)} \Lambda \}.
\]
Since the vectors of $H_p$ are linearly independent and have span which intersects $T_{(x,\xi)} \Lambda$ trivially, we necessarily have $t' = 0$. Hence,
\[
	F' = \{(0,x',\xi') : (x',\xi') \in T_{(x,\xi)} \Lambda\} = T_{(0,x,\xi)} F,
\]
as desired.
\end{proof}

%Next we prepare for the proof of part (3). We let
%\[
%	E_{\tau} = \{(0,x,\xi) \in F : p(x,\xi) = \tau \}
%\]
%denote the excess fiber over each point $(0,\tau) \in \mathcal C \circ \Lambda$. By \cite{DG} or \eqref{integral over excess fibers}
%\[
%	\sigma_{\mathcal C} \circ \sigma_{\Lambda}\big|_{(0,\tau)} = \int_{E_\tau} \sigma_{\mathcal C} \boxtimes \sigma_\Lambda 
%\]
%where at each $(0,x,\xi) \in F$, the object $\sigma_{\mathcal C} \boxtimes \sigma_\Lambda$  belongs to $|T_{(0,x,\xi)} E_\tau| \otimes |T_{(0,\tau)} (\mathcal C \circ \Lambda)|^{1/2}$ fiberwise. For now, we establish what we require from $\sigma_{\mathcal C} \boxtimes \sigma_\Lambda$ to prove part (3). We do this by selecting convenient coordinates for $\Lambda$.

In preparation for the proof of part (3), we select some convenient local coordinates.

\begin{lemma}\label{lambda coordinates lem} There exist local coordinates $\lambda = (\lambda_1,\ldots,\lambda_n)$ about each point in $\Lambda$ with respect to which
\[
	dp = \begin{bmatrix}
		I & 0
	\end{bmatrix},
\]
where $I$ is the $k \times k$ identity matrix and $0$ is the $k \times (n - k)$ zero matrix.
\end{lemma}

\begin{proof}
	We first claim the restriction of $p$ to $\Lambda$ has surjective differential, i.e. that the rows of the matrix with entries
	\[
		dp_i\left(\frac{\partial}{\partial \lambda_j}\right) \qquad \text{$i = 1,\ldots,k$ and $j = 1,\ldots,n$}
	\]
	are linearly independent. Given coefficients $t_1,\ldots,t_k$ and the symplectic $2$-form $\omega$ on $Y$,
	\[
		\sum_{i = 1}^k t_i p_i \left(\frac{\partial}{\partial \lambda_j}\right) = \omega\left( \frac{\partial}{\partial \lambda_j}, \sum_{i = 1}^k t_i H_{p_i} \right) \qquad \text{ for each $j = 1,\ldots,n$.}
	\]
	Since $\Lambda$ is Lagrangian, this quantity vanishes if and only if the linear combination lies in $T\Lambda$, which may only ever happen when $t_i = 0$ for each $i$ by hypothesis. This proves our claim.
	
	Next, fix $\tau$ and note $p^{-1}(\tau)$ is a smooth, codimension-$k$ submanifold of $\Lambda$. Select local coordinates $\lambda_{k+1},\ldots,\lambda_n$ of $p^{-1}(\tau)$ and extend them to coordinates
	\[
		p_1,\ldots,p_k,\lambda_{k+1}, \ldots, \lambda_n
	\]
	of $\Lambda$. The lemma follows.
\end{proof}

For each $\tau$, we denote the excess fiber over $(0,\tau) \in \mathcal C \circ \Lambda$ as
\[
	E_\tau = \{(0,x,\xi) \in F : p(x,\xi) = \tau\}.
\]
The clean composition calculus of Duistermaat and Guillemin yields a composite symbol on $\mathcal C \circ \Lambda$ given by
\[
	\sigma_{\mathcal C} \circ \sigma_{\Lambda} = \int_{E_\tau} \sigma_{\mathcal C} \boxtimes \sigma_\Lambda,
\]
where $\sigma_{\mathcal C} \boxtimes \sigma_{\Lambda}$ is a smooth object belonging to $|T_{(0,x,\xi)} E_\tau| \otimes |T_{(0,\tau)} (\mathcal C \circ \Lambda)|^{1/2}$ fiberwise and is determined by the procedure summarized in Appendix \ref{SYMBOL CALCULUS}. 

With respect to the coordinates of the lemma, we have
\[
	d_p(\lambda) = |d\lambda_{k+1} \ldots d\lambda_n|.
\]
Hence to prove (3), it suffices to show that
\[
	\sigma_{\mathcal C} \boxtimes \sigma_\Lambda = b(0,\lambda) a(\lambda) |d\lambda_{k+1} \cdots d\lambda_n| |d\tau|^{1/2},
\]
i.e. $\sigma_{\mathcal C} \boxtimes \sigma_\Lambda$ evaluates to $b(0,\lambda) a(\lambda)$ on the pair of bases,
\[
	\left( \left( \frac{\partial}{\partial \lambda_1}, \ldots, \frac{\partial}{\partial \lambda_n} \right), \left( \frac{\partial}{\partial \tau_1}, \cdots, \frac{\partial}{\partial \tau_k} \right) \right).
\]

Fix $(x,\xi) \in \Lambda$ and $\tau = p(x,\xi)$. We will identify the $t = 0$ component of $F$ with $\Lambda$. Since we will be exclusively working in the linear category, we take $\mathcal C$, $\Lambda$, and $F$ to stand in for their respective tangent spaces $T_{(0,p(x,\xi);x,\xi)} \mathcal C$, $T_{(x,\xi)} \Lambda$, and $T_{(x,\xi)} F$. We also take $Y$ to stand in for the symplectic space $T_{(x,\xi)}T^*Y$ endowed with symplectic $2$-form $\omega$. We let $E$ denote the tangent space $T_{(x,\xi)}E_\tau = \{(x',\xi') \in T_{(x,\xi)} \Lambda : dp(x',\xi') = 0\}$ of the excess fiber.

As summarized in Appendix \ref{SYMBOL CALCULUS}, we have two maps,
\begin{equation}\label{my alpha}
	\alpha : F \to \mathcal C \circ \Lambda \qquad \text{ where } \qquad (x',\xi') \mapsto (0, dp(x',\xi')),
\end{equation}
and
\[
	\beta : \mathcal C \times \Lambda \to Y \qquad \text{ where } \qquad ((a'; b'), c') \mapsto b' - c'.
\]
Associated with these maps are two exact sequences,
\begin{equation}\label{my short exact sequence}
	0 \longrightarrow F \longrightarrow \mathcal C \times \Lambda \overset{\beta}{\longrightarrow} Y \longrightarrow \coker \beta \longrightarrow 0
\end{equation}
and
\begin{equation}\label{my long exact sequence}
	0 \longrightarrow E \longrightarrow F \overset{\alpha}{\longrightarrow} \mathcal C \circ \Lambda \longrightarrow 0.
\end{equation}
These, along with the pairing of $E$ and $\coker \tau$ by the symplectic form on $Y$, yield a linear isomorphism $|\mathcal C \times \Lambda|^{1/2} \simeq |E| \otimes |\mathcal C \circ \Lambda|^{1/2}$ which takes $\sigma_{\mathcal C} \otimes \sigma_\Lambda \mapsto \sigma_{\mathcal C} \boxtimes \sigma_\Lambda$. To make this isomorphism explicit, we write down some bases for the spaces involved.

\begin{lemma}\label{basis lem}
	There exists a basis $\mathbf y = (\mathbf u, \mathbf v, H_p, \mathbf w)$ for $Y$ with the following properties.
	\begin{enumerate}
		\item Given coordinates $(\lambda_1,\ldots,\lambda_n)$ as in Lemma \ref{lambda coordinates lem},
		\[
			\mathbf u = \left( \frac{\partial}{\partial \lambda_1}, \ldots, \frac{\partial}{\partial \lambda_k} \right) \qquad \text{ and } \qquad \mathbf v = \left( \frac{\partial}{\partial \lambda_{k+1}}, \ldots, \frac{\partial}{\partial \lambda_n} \right).
		\]
		Hence, $\mathbf v$ is a basis for $E$ and $(\mathbf u, \mathbf v)$ is a basis for $F$.
		\item The image of $\mathbf u$ through $\alpha$ is
		\[
			\alpha(\mathbf u) = \left( \frac{\partial}{\partial \tau_1}, \ldots, \frac{\partial}{\partial \tau_k} \right).
		\]
		\item $\omega(\mathbf u,H_p) = I$, where $I$ denotes the $k \times k$ identity matrix.
		\item $\mathbf y = (\mathbf u, \mathbf v, H_p, \mathbf w)$ is a symplectic basis for $Y$. Specifically, $\omega(\mathbf u,\mathbf w) = 0$, $\omega(H_p, \mathbf  w) = 0$, and $\omega(\mathbf v, \mathbf w) = I$ where $I$ denotes the $(n-k) \times (n-k)$ identity matrix.
	\end{enumerate}
\end{lemma}

\begin{proof}
	(1) is a definition. Lemma \ref{lambda coordinates lem} tells us $dp(\mathbf u) = I$, which yields (2). (3) similarly follows since $\omega(\mathbf u, H_p) = dp(\mathbf u)$. For (4), we first note that $(\mathbf u, \mathbf v, H_p)$ is an incomplete symplectic basis for $Y$. This follows from part (3), from $\omega(\mathbf u, \mathbf v) = 0$ since $\Lambda$ is Lagrangian, and from that $\omega(\mathbf v, H_p) = 0$ also by Lemma \ref{lambda coordinates lem}. Proposition 21.1.3 of \cite{HormanderIII} ensures it can be completed.
\end{proof}

We are now ready to prove part (3) of Lemma \ref{symbol composition lem}.

\begin{proof}[Proof of Lemma \ref{symbol composition lem} part (3)]
	Recall, our goal is to show $\sigma_{\mathcal C} \boxtimes \sigma_{\Lambda}$ evaluates to $a(0,\lambda)b(\lambda)$ at the pair of bases
	\[
		\left( \mathbf v, \left( \frac{\partial}{\partial \tau_1}, \ldots, \frac{\partial}{\partial \tau_k} \right) \right).
	\]
	As described in the appendices, we start by identifying the isomorphism
	\[
		|\mathcal C \times \Lambda|^{1/2} \simeq |F|^{1/2} \otimes |\coker \beta|^{-1/2}
	\]
	induced by the exact sequence \eqref{my long exact sequence}. Let $(\mathbf u, \mathbf v)$ be the basis for $F$ as in Lemma \ref{basis lem}. The image of $(\mathbf u, \mathbf v)$ via the map $F \to \mathcal C \times \Lambda$ is
	\begin{align}
		\label{starting basis 1} &\left(\left( \frac{\partial}{\partial \tau_i}; \frac{\partial}{\partial \lambda_i} \right), \frac{\partial}{\partial \lambda_i} \right), && i = 1,\ldots,k, \\
		\nonumber &\left(\left( 0; \frac{\partial}{\partial \lambda_{i}} \right),\frac{\partial}{\partial \lambda_{i}} \right), && i = k+1,\ldots,n.
	\end{align}
	We can extend this to a basis for the product $\mathcal C \times \Lambda$ by introducing elements
	\begin{align}
		\label{basis extension 1} &\left( \left( \frac{\partial}{\partial \tau_i}; \frac{\partial}{\partial \lambda_i} \right), 0 \right) && i = 1,\ldots,k, \\
		\nonumber &\left( \left( 0; \frac{\partial}{\partial \lambda_{i}} \right), 0 \right) && i = k+1,\ldots,n,\\
		\nonumber &\left( \left( \frac{\partial}{\partial t_i}; H_{p_i} \right), 0 \right) && i = 1,\ldots,k.
	\end{align}
	After a sequence of determinant-$\pm 1$ operations, we can rewrite this basis as
	\begin{align*}
		&\left( \left( \frac{\partial}{\partial \tau_i}; \frac{\partial}{\partial \lambda_i} \right), 0 \right) && i = 1,\ldots,k, \\
		&\left( \left( 0; \frac{\partial}{\partial \lambda_{i}} \right), 0 \right) && i = k+1,\ldots,n, \\
		&\left( \left( \frac{\partial}{\partial t_i}; H_{p_i} \right), 0 \right) && i = 1,\ldots,k,\\
		& \left( 0, \frac{\partial}{\partial \lambda_i} \right) && i = 1,\ldots, n.
	\end{align*}
	Recall $\sigma_{\mathcal C} = a(t,\lambda) |dt \, d\lambda|^{1/2}$ and $\sigma_\Lambda = b(\lambda) |d\lambda|^{1/2}$, hence $\sigma_{\mathcal C} \otimes \sigma_\Lambda$ evaluates to $b(0,\lambda) a(\lambda)$ on the basis consisting of elements in \eqref{starting basis 1} and \eqref{basis extension 1}. The image of the extension \eqref{basis extension 1} through $\beta$ is $(\mathbf u, \mathbf v, H_p)$, and since the sequence is exact, is a basis for the image of $\beta$. As per Lemma \ref{basis lem}, we extend this to a symplectic basis for $Y$ by adding in some vectors $\mathbf w$. Hence, we have identified an object in $|F|^{1/2} \otimes |\coker \beta|^{-1/2}$ which assigns the value $a(0,\lambda) b(\lambda)$ to the pair of bases $((\mathbf u, \mathbf v), \mathbf w)$.
	
	Next, by (4) of Lemma \ref{basis lem} and Appendix \ref{LINEAR DENSITIES}, we have an isomorphism $|\coker \beta|^{-1/2} \simeq |E|^{1/2}$ which assigns to both objects the same value when evaluated on bases $\mathbf w$ and $\mathbf v$, respectively. This maps our object in $|F|^{1/2} \otimes |\coker \beta|^{-1/2}$ to the object in $|F|^{1/2} \otimes |E|^{1/2}$ which assigns the value $a(0,\lambda)b(\lambda)$ to the pair of bases $((\mathbf u, \mathbf v), \mathbf w)$.
	
	Finally, the exact sequence \eqref{my short exact sequence} yields an identification
	\[
		|F|^{1/2} \simeq |E|^{1/2} \otimes |\mathcal C \circ \Lambda|^{1/2},
	\]
	which assigns the same value to both objects when evaluated on the basis $(\mathbf u, \mathbf v)$ and the pair of bases $(\mathbf v, (\frac{\partial}{\partial \tau_1} , \ldots, \frac{\partial}{\partial \tau_k} ))$, respectively. This maps our object in $|F|^{1/2} \otimes |E|^{1/2}$ to the object $\sigma_{\mathcal C} \boxtimes \sigma_{\Lambda}$ in $|E| \otimes |\mathcal C \circ \Lambda|^{1/2}$ which assigns the value $a(0,\lambda) b(\lambda)$ to the pair of bases $(\mathbf v, (\frac{\partial}{\partial \tau_1} , \ldots, \frac{\partial}{\partial \tau_k} ))$.
\end{proof}

%%%%%%%%%%%%%%%%%%%%%%%%%%%%%%%%%%%%%%%%%
% Proof of main composition proposition %
%%%%%%%%%%%%%%%%%%%%%%%%%%%%%%%%%%%%%%%%%

\section{The Symbol Calculations of Propositions \ref{prop first composition} and \ref{main composition prop}} \label{SYMBOLS}

Proposition \ref{prop first composition} is a direct consequence of Lemma \ref{symbol composition lem trans} after taking $p$ in the lemma as
\begin{equation*}\label{p label}
	p(x,\xi) = (p(x_1,\xi_1), p(x_2,\xi_2), p(x_3,\xi_3))
\end{equation*}
as usual. We note that $p$ is not smooth if any of $\xi_1,\xi_2,$ or $\xi_3$ vanish, but this case is excluded since $p(x,\xi)$ belongs to a triangle-good cone and hence has positive components.

Before proceeding with the proof of Proposition \ref{main composition prop}, we require some facts about geodesic normal coordinates. Geodesic normal coordinates about a point $q \in M$ are coordinates $(x_1,\ldots,x_n)$ for which $q$ corresponds to the origin and the metric $g$ satisfies, among other things,
\[
	g_{ij}(0) = I \qquad \text{ and } \qquad \frac{\partial}{\partial x_k} g_{ij}(0) = 0
\]
for all $i,j,k$. The half Laplacian $P = \sqrt{-\Delta}$ has principal symbol
\[
	p(x,\xi) = \sqrt{\sum_{i,j} g^{ij}(x) \xi_i \xi_j}
\]
in canonical local coordinates $(x_1,\ldots,x_n,\xi_1,\ldots,\xi_n)$ of $T^*M$. We note, in particular, that
\[
	p(0,\xi) = |\xi| \qquad \text{ and } \qquad \frac{\partial}{\partial x_k} p(0,\xi) = 0
\]
for each $k$. We recall the local form \eqref{local hamiltonian} of the Hamilton vector field $H_p$, which at $(0,\xi)$ in geodesic normal coordinates is written conveniently as
\[
	H_p(0,\xi) = \sum_k \frac{\xi_k}{|\xi|} \frac{\partial}{\partial x_k}.
\]

\begin{proof}[Proof of Proposition \ref{main composition prop}]
To apply Lemma \ref{symbol composition lem}, we must verify each of
\[
	H_p(x_1,\xi_1) \oplus 0 \oplus 0, \quad 0 \oplus H_p(x_1,\xi_2) \oplus 0, \quad 0 \oplus 0 \oplus H_p(x_1,\xi_3)
\]
on $T^*M^3$ are linearly independent (which is automatic) and have span which intersects the tangent space to $\dot N^*\Delta$ trivially. To see this, suppose 
\[
	(t_1' H_{p}(x_1,\xi_1), t_2' H_{p}(x_1,\xi_2), t_3' H_{p}(x_1,\xi_3)) \in T\dot N^*\Delta.
\]
The pushforward of the vector through $T^*M^3 \to M^3$ is
\[
	(t_1' \xi_1/|\xi_1|, t_2' \xi_2/|\xi_2|, t_3' \xi_3/|\xi_3|)
\]
in geodesic normal coordinates about $x_1$, and should lie in $T\Delta$, meaning that
\[
	t_1' \xi_1/|\xi_1| = t_2' \xi_2/|\xi_2| = t_3' \xi_3/|\xi_3| = 0.
\]
But, since $\xi_1,\xi_2,$ and $\xi_3$ specify the sides of a nondegenerate triangle, any two of them are linearly independent. Hence, $t_1' = t_2' = t_3' = 0$ as desired.

Part (1) of Lemma \ref{symbol composition lem} shows us $\mathcal C \circ \dot N^*\Delta$ has an isolated component in $T_0^*\R^3$. However, we can say more. If there is a triple of times $t = (t_1,t_2,t_3)$ which are not all zero, and for which $e^{t_1 H_p}(x_1,\xi_1)$, $e^{t_2 H_p}(x_1,\xi_2)$, and $e^{t_3H_p}(x_1,\xi_3)$ project to the same base point, then at least one of $|t_1|$, $|t_2|$, or $|t_3|$ is at least as large as $\inj M$. Hence, there are no other components in $\mathcal C \circ \dot N^*\Delta$ which have a $t$ parameter in the cube $(-\inj M, \inj M)^3$.

Part (2) of the lemma tells us the excess of the composition $\mathcal C \circ \dot N^*\Delta$ over the $t = 0$ component, namely
\[
	e = 3n - 3.
\]
Hence, we have
\begin{align*}
	\ord S \circ \delta_\Delta &= \ord S + \ord \delta_\Delta + \frac{e}{2} \\
	&= \frac{n-3}{4} + \frac{n}{4} + \frac{3n - 3}{2} \\
	&= 2n - \frac{9}{4},
\end{align*}
as desired.

Before proceeding with the remainder of the proof of Proposition \ref{main composition prop}, we outline two standard calculus identities for Leray densities.

\begin{proposition}\label{Leray calculus}
Let $F : \R^n \to \R^d$ be smooth and have surjective differential on $F^{-1}(y)$ for some fixed $y \in \R^d$. We then have the following for any continuous function $f$ on $\R^n$.
\begin{enumerate}
	\item (Change of variables) Given a diffeomorphism $\Phi : \R^n \to \R^n$, we have
	\[
		\int_{\R^n} f(x) \delta(F(x) - y) \, dx = \int_{\R^n} f(\Phi(x)) \delta(F \circ \Phi(x) - y) |\det d\Phi(x)| \, dx.
	\]
	\item (Fubini's Theorem) Let $x = (x',x'')$ where $x' \in \R^m$ and $x'' \in \R^{n-m}$. Similarly, let $y = (y',y'')$ where $y' \in \R^k$ and $y'' \in \R^{d-k}$. Suppose $F(x) = (G(x'), H(x',x''))$. Then, $dG$ is surjective on all $x'$ for which $F(x',x'') = y$ for some $x''$, and $d_{x''}H$ is surjective on $F^{-1}(y)$. Furthermore,
	\begin{multline*}
		\int_{\R^n} f(x) \delta(F(x) - y) \, dx \\
		= \int_{\R^m} \left( \int_{\R^{n-m}} f(x',x'') \delta(H(x',x'') - y'') \, dx'' \right) \delta(G(x') - y') \, dx'.
	\end{multline*}
\end{enumerate}
\end{proposition}

The proof is elementary, so we omit it. These identities also readily extend to manifolds (recall the discussion of Leray densities on manifolds from Section \ref{STATEMENT OF RESULTS}), but again leave the details to the reader. We will make use of this proposition repeatedly and without reference.

The only thing left to prove of Proposition \ref{main composition prop} is the composition of the symbol. We have by Theorem 5.4 of \cite{DG}
\[
	\sigma_{S \circ \delta_\Delta} = (2\pi i)^{-e/2} \sigma_{S} \circ \sigma_{\delta_\Delta},
\]
where by part (3) of Lemma \ref{symbol composition lem}, we write locally
\[
	\sigma_{S} \circ \sigma_{\delta_\Delta} = (2\pi)^{-n/2 + 3/4} \left(\int_{p(x,\xi) = \tau} |g(x_1)|^{-1/2} a(p(x,\xi)) \, d_p \right) \, |d\tau|^{1/2}.
\]
Here, $(x,\xi)$ is a stand-in for $(x_1,x_1,x_1,-\xi_2-\xi_3,\xi_2,\xi_3)$, and the integral is over the variables $x_1, \xi_2,$ and $\xi_3$ for which $p(x,\xi) = \tau$, with respect to the Leray density. We write the integral above as
\begin{equation}\label{almost there}
	a(\tau) \int_M |g(x_1)|^{1/2} \left(\iint |g(x_1)|^{-1} \delta(p(x,\xi) - \tau) \, d\xi_2 \, d\xi_3 \right) \, dx_1.
\end{equation}
By the change of variables formula for covectors, the double integral is invariant under a change of variables in $M$. Hence, it suffices to compute the double integral for fixed $x_1$ with respect to geodesic local coordinates and then integrate. In such coordinates,
\[
	p(x,\xi) = (|\xi_2 + \xi_3|, |\xi_2|, |\xi_3|) \qquad \text{ and } \qquad |g(x_1)|^{-1} = 1.
\]
Recalling $F$ from \eqref{def F}, the double integral above is written
\[
	\iint \delta(F(\xi_2,\xi_3) - (\tau_2,\tau_3,\tau_1)) \, d\xi_2 \, d\xi_3 = \vol F^{-1}(\tau)
\]
and Proposition \ref{main composition prop} follows.
\end{proof}

%%%%%%%%%%%%%%%%%%%%%%%%%%
% Proof of Leray measure %
%%%%%%%%%%%%%%%%%%%%%%%%%%

\section{Computation of the Leray Measure in Proposition \ref{prop leray measure}}

In this section, we will make liberal use of Proposition \ref{Leray calculus} and its analogue on manifolds. Recall the Leray measure of $F^{-1}(a,b,c)$,
\[
	\int_{F^{-1}(a,b,c)} d_F = \iint \delta(F(\xi,\eta) - (a,b,c)) \, d\xi \, d\eta.
\]
We first write this quantity as
\[
	\int \delta(|\xi| - a) \left( \int \delta((|\eta|,|\xi + \eta|) - (b,c)) \, d\eta \right) \, d\xi,
\]
By an orthogonal change of variables, we see the inner integral is constant in $\xi$, so we fix a representative $\xi = ae_1$, where $e_1 = (1,0,\ldots,0)$ is the first standard basis element. The integral is then
\[
	\vol(S^{n-1}) a^{n-1} \int \delta((|\eta|, |ae_1 + \eta|) - (b,c)) \, d\eta.
\]
Next, we change variables $\eta = r \omega$ where $r > 0$ and $\omega \in S^{n-1}$. We then have
\begin{multline*}
	\vol(S^{n-1}) a^{n-1} \iint \delta((r,|ae_1 + r\omega|) - (b,c)) r^{n-1} \, dr \, dV_{S^{n-1}}(\omega) \\
	= \vol(S^{n-1}) a^{n-1} \int \delta(r - b) r^{n-1} \left( \int_{|ae_1 + r\omega| = c} \delta(|ae_1 + r \omega| - c) \, dV_{S^{n-1}}(\omega) \right) \, dr \\
	= \vol(S^{n-1}) a^{n-1} b^{n-1} \int \delta(|ae_1 + b\omega| - c) \, dV_{S^{n-1}}(\omega)
\end{multline*}
where $dV_{S^{n-1}}$ is the standard volume density on $S^{n-1}$. Next, we perform a change of variables for the inner integral, namely $\omega = \cos \theta e_1 + \sin \theta \omega'$ where $0 < \theta < \pi$ and $\omega' \in 0 \times S^{n-2}$ lives in the unit sphere in the orthogonal complement to $e_1$. Written locally, the Riemannian metric on $S^{n-1}$ with respect to $(\theta, \omega')$ coordinates is
\[
	g_{S^{n-1}}(\theta,\omega') = \begin{bmatrix}
		1 & 0 \\
		0 & \sin^2(\theta) g_{S^{n-2}}(\omega')
	\end{bmatrix},
\]
and hence
\begin{align*}
	dV_{S^{n-1}}(\omega) &= |g_{S^{n-1}}(\theta,\omega')|^{1/2} \, d\theta \, d\omega \\
	&= \sin^{n-2}(\theta) |g_{S^{n-2}}(\omega')|^{1/2} \, d\theta \, d\omega'\\
	&= \sin^{n-2}(\theta) \, d\theta \, dV_{S^{n-2}}(\omega').
\end{align*}
Hence, $\vol F^{-1}(a,b,c)$ now reads
\begin{multline*}
	 \vol(S^{n-1}) a^{n-1} b^{n-1} \iint \delta(|(a + b\cos \theta)e_1 + b\sin\theta \omega'| - c) \sin^{n-2}(\theta) \, d\theta \, dV_{S^{n-2}}(\omega') \\
	 = \vol(S^{n-1}) a^{n-1} b^{n-1} \int_{S^{n-2}} \left( \int \delta(|(a + b\cos \theta)e_1 + b\sin\theta \omega'| - c) \sin^{n-2}(\theta) \, d\theta \right) \, dV_{S^{n-2}}(\omega').
\end{multline*}
Again, the inner integral is constant with respect to $\omega'$, and so we select a candidate $\omega' = e_2$ for the inner integral and integrate out the $\omega'$ to obtain
\[
	 \vol(S^{n-1}) \vol(S^{n-2}) a^{n-1} b^{n-1} \int \delta(|(a + b\cos \theta)e_1 + b\sin\theta e_2| - c) \sin^{n-2}(\theta) \, d\theta.
\]

Now that we have reduced things down to one dimension, elementary geometry and calculus take us the rest of the way. We have
\[
	|(a + b\cos \theta)e_1 + b\sin\theta e_2| = \sqrt{a^2 + b^2 + 2ab \cos \theta}.
\]
Recall, $\theta$ was constructed to be angle between $\xi$ and $\eta$. That is, $\theta$ is the exterior angle where sides of lengths $a$ and $b$ of a triangle meet. That
\[
	\sqrt{a^2 + b^2 + 2ab \cos \theta} = c
\]
means the third side has length $c$. This may happen at most once for $0 < \theta < \pi$. Furthermore, since $0 < c < a + b$ by hypothesis, this happens for exactly one such value of $\theta$. The integral above is a valuation of the integrand at this angle times a factor
\[
	\left| \frac{d}{d\theta} \sqrt{a^2 + b^2 + 2ab \cos \theta} \right|^{-1} = \frac{c}{ab \sin \theta}.
\]
Putting everything together, we obtain a Leray measure of
\begin{multline*}
	 \vol(S^{n-1}) \vol(S^{n-2}) a^{n-1} b^{n-1} \frac{c}{ab \sin \theta} \sin^{n-2} \theta \\
	 = \vol(S^{n-1}) \vol(S^{n-2}) a^{n-2} b^{n-2} c \sin^{n-3} \theta \\
	 = \vol(S^{n-1}) \vol(S^{n-2}) abc (2 \area(a,b,c))^{n-3},
\end{multline*}
where we have used $ab \sin \theta = 2 \area(a,b,c)$ in the last line.

%%%%%%%%%%%%%%%%%%%%%%%%%%%%%%%
% Proof of degenerate theorem %
%%%%%%%%%%%%%%%%%%%%%%%%%%%%%%%

\section{Proof of Corollary \ref{main theorem triangle-degenerate}}

We have
\begin{align*}
	\int_{-\infty}^\infty &\rho * \mu(\tau_1,\tau_2,\tau_3) \, d\tau_3\\
	&= \sum_{i,j,k} \int_{-\infty}^\infty \chi(\tau_1 - \lambda_i) \chi(\tau_2 - \lambda_j) \chi(\tau_3 - \lambda_k) |\langle e_i e_j, e_k \rangle|^2 \, d\tau_3 \\
	&= \sum_{i,j,k} \chi(\tau_1 - \lambda_i) \chi(\tau_2 - \lambda_j) |\langle e_i e_j, e_k \rangle|^2 \\
	&= \sum_{i,j} \chi(\tau_1 - \lambda_i) \chi(\tau_2 - \lambda_j) \| e_i e_j \|_{L^2(M)}^2 \\
	&= \int_M \left(\sum_i \chi(\tau_1 - \lambda_i) |e_i(x)|^2 \right) \left( \sum_j \chi(\tau_2 - \lambda_j) |e_j(x)|^2 \right) \, dV_M(x).
\end{align*}
We have uniform asymptotics (e.g. by Proposition 29.1.2 of \cite{HormanderIV})
\[
	\sum_i \chi(\tau_1 - \lambda_i) |e_i(x)|^2 = (2\pi)^{-n} (\vol S^{n-1}) \tau_1^{n-1} + O(\tau_1^{n-2})
\]
and similarly for the sum over $j$. Hence,
\[
	\int_{-\infty}^\infty \rho * \mu(\tau_1,\tau_2,\tau_3) \, d\tau_3 = (\vol M)(\vol S^{n-1})^2 \tau_1^{n-1} \tau_2^{n-1} + O((\tau_1 + \tau_2)\tau_1^{n-2}\tau_2^{n-2}).
\]
Theorem \ref{main theorem triangle-degenerate} follows provided we can show
\begin{equation*} \label{leray equation}
	\int_{|\tau_1 - \tau_2|}^{\tau_1 + \tau_2} \vol F^{-1}(\tau_1,\tau_2,\tau_3) \, d\tau_3 = (\vol S^{n-1})^2 \tau_1^{n-1} \tau_2^{n-1}.
\end{equation*}
But, this holds by Proposition \ref{prop leray measure} and an explicit computation.

%%%%%%%%%%%%%%%%%%
% Begin Appendix %
%%%%%%%%%%%%%%%%%%

\appendix

%%%%%%%%%%%%%
% Densities %
%%%%%%%%%%%%%

\section{The $\alpha$-Density Formalism} \label{DENSITIES}

We review the most relevant parts of the $\alpha$-density formalism as found in Guillemin and Sternberg's \emph{Semi-Classical Analysis} \cite{GS}.

\subsection{$\alpha$-densities on vector spaces}\label{LINEAR DENSITIES}

We begin with the basic definition.

\begin{definition}\label{def alpha density} Let $V$ be a finite-dimensional real vector space and $\alpha$ any real number. An \emph{$\alpha$-density} on $V$ is a complex-valued function $f$ on the set of bases for $V$ satisfying the change of basis formula
\[
	f(Tv_1,\ldots,Tv_n) = |\det T|^\alpha f(v_1,\ldots,v_n)
\]
for any ordered basis $v_1,\ldots,v_n$ of $V$ and invertible linear operator $T$ on $V$. The set of $\alpha$-densities on $V$ is denoted $|V|^\alpha$. We refer to the $1$-densities simply as densities, and write $|V|^1$ as $|V|$.
\end{definition}

Note $|V|^\alpha$ is a one-dimensional vector space over the complex numbers, or complex line. The space of $0$-densities admits a natural identification with the complex numbers by its evaluation on any basis. By convention, we identify $|0|^\alpha$ with $\C$ by viewing its elements as functions which may be evaluated only at the empty basis.

What follows is a catalogue of only the most relevant manipulations of $\alpha$-densities for the symbol calculus. But first, we introduce some helpful notation. Given a vector space $V$, we use a boldface letter $\mathbf v = (v_1,\ldots,v_n)$ to denote an ordered list of vectors in $V$. If $T : V \to W$ is a linear map, we let $T\mathbf v = (Tv_1, \ldots, Tv_n)$.

We begin with the linear isomorphism
\[
	|V|^\alpha \otimes |V|^\beta \simeq |V|^{\alpha + \beta}
\]
given by the map $\sigma \otimes \tau \mapsto \sigma \tau$. Specifically, we mean $\sigma \otimes \tau$ is mapped to an $\alpha$-density on $V$ which assigns to a basis $\mathbf v$ the value $\sigma(\mathbf v) \tau(\mathbf v)$. One quickly checks that $\sigma \tau \in |V|^{\alpha + \beta}$ and that the map is linear and nonzero. (Recall, the tensor product of two complex lines is again a complex line.)

Next, we discuss isomorphisms between $\alpha$-densities induced by exact sequences. Consider first
\[
	0 \longrightarrow X \overset T\longrightarrow Y \overset S\longrightarrow Z \longrightarrow 0.
\]
We have a linear isomorphism $|Y|^\alpha \simeq |X|^\alpha \otimes |Z|^\alpha$ as follows. Fix $\alpha$-densities $\sigma_X$ and $\sigma_Z$ on $X$ and $Z$, respectively. Secondly, fix a basis $\mathbf x$ on $X$, and extend $T\mathbf x$ to a basis $(T\mathbf x, \mathbf y)$ of $Y$. Thirdly, note $S \mathbf y$ is a basis for $Z$. Our isomorphism is given explicitly by
\[
	\sigma_Y(T\mathbf x, \mathbf y) = \sigma_X(\mathbf x) \sigma_Z(S\mathbf y).
\]
One checks that $\sigma_Y$ satisfies Definition \ref{def alpha density} under a change of basis in $\mathbf x$ and $\mathbf y$. We have
\[
	|X \times Z|^\alpha \simeq |X|^\alpha \otimes |Z|^\alpha
\]
as a corollary.
%Note, the Leray density on $F^{-1}(0)$ is obtained in this way via
%\[
%	0 \longrightarrow \ker dF \longrightarrow \R^n \overset{dF}{\longrightarrow} \R^k \longrightarrow 0.
%\]

Given a longer exact sequence such as
\[
	0 \longrightarrow W \overset R \longrightarrow X \overset T \longrightarrow Y \overset S \longrightarrow Z \longrightarrow 0,
\]
we have a linear isomorphism
\[
	|X|^\alpha \simeq |W|^\alpha \otimes |Y|^\alpha \otimes |Z|^{-\alpha}
\]
given explicitly as follows. Let $\sigma_W$ and $\sigma_Y$ be half-densities on $W$ and $Y$, respectively, and let $\tilde \sigma_Z$ be a $-\frac12$-density on $Z$. Let $\mathbf w$ be a basis for $W$, let $(R\mathbf w, \mathbf x)$ be a basis for $X$, let $(T\mathbf x, \mathbf y)$ be a basis for $Y$, and note $S\mathbf y$ is a basis for $Z$. Then, the isomorphism is given explicitly by
\[
	\sigma_X(R\mathbf w, \mathbf x) = \sigma_W(\mathbf w) \sigma_Y(T\mathbf x, \mathbf y) \tilde \sigma_Z(S\mathbf y).
\]
Again, one quickly checks that the value of the right side is independent of the choice of $\mathbf y$ and that $\sigma_X$ satisfies the correct change of basis formula in $\mathbf w$ and $\mathbf x$.

Finally, if $\omega$ is a nondegenerate bilinear form on $X \times Y$, we have an isomorphism
\[
	|Y|^{1/2} \simeq |X|^{-1/2}
\]
given explicitly as follows. Let $\tilde \sigma_X$ be a $-\frac12$-density on $X$, then we have
\[
	\sigma_Y(\mathbf y) = \tilde \sigma_X(\mathbf x)
\]
for a pair of bases $\mathbf x = (x_1,\ldots,x_n)$ and $\mathbf y = (y_1,\ldots,y_n)$ for which $\omega(x_i, y_j) = \delta_{ij}$. Again, one checks $\sigma_Y$ defined in this way satisfies the correct change of basis formula.

\subsection{$\alpha$-densities on manifolds}

Let $X$ be a smooth manifold. The $\alpha$-density bundle over $X$, which we again denote $|X|^\alpha$, is the disjoint union
\[
	\bigsqcup_{x \in X} |T_x X|^\alpha.
\]
This has the structure of a smooth complex line bundle over $X$. Given local coordinates $x = (x_1,\ldots,x_n)$ of an open neighborhood $U \subset X$, the $\alpha$-density bundle admits a local trivialization
\[
	\begin{split}
	U \times \C &\to |X|^\alpha\\
	(x,z) &\mapsto z |dx|^\alpha
	\end{split}
\]
where $|dx|^\alpha$ is the $\alpha$-density which assigns the value $1$ to the local coordinate frame
\[
	\left( \frac{\partial}{\partial x_1}, \ldots, \frac{\partial}{\partial x_n} \right).
\]
The smooth $\alpha$-densities are precisely the smooth sections of $|X|^\alpha$. We denote the space of smooth $\alpha$-densities on $X$ as $C^\infty(|X|^\alpha)$. Each can be locally expressed as
\[
	f(x) |dx|^\alpha
\]
for some smooth function $f$ on $X$. If $y = \kappa(x)$ is a smooth change of coordinates, then we have a change of coordinates rule
\[
	f(x) |dx|^\alpha = |\det d\kappa|^{\alpha} f(\kappa(y)) |dy|^\alpha.
\]

Next, we discuss integrals of densities. In the case where $\alpha = 1$, the integral of a smooth density is invariant under change of coordinates and hence invariantly defined. Furthermore, if $f \in C^\infty(|X|^\alpha)$ and $g \in C^\infty(|X|^\beta)$, then $fg \in C^\infty(|X|^{\alpha + \beta})$. If $\alpha = \frac12$, for example, we have a well-defined inner product pairing of half-densities
\[
	\langle f, g \rangle = \int_X f \overline g.
\]
We can complete the smooth half-densities under the corresponding norm $\| f\| = \sqrt{\langle f, f \rangle}$ to obtain the \emph{intrinsic $L^2$-space of $X$}. The pairing also allows us to define space $\mathcal D'(|X|^\frac12)$ of \emph{half-density distributions} as the topological dual to the smooth, compactly-supported half-densities $C^\infty_c(|X|^\frac12)$. Unlike standard distributions, half-density distributions change variables like half-densities, and hence there is a coordinate-invariant embedding $C_c^\infty(|X|^\frac12) \to \mathcal D'(|X|^\frac12)$ of test half-densities into half-density distributions.

%%%%%%%%%%%%%%%%%%%%%%%%
% Composition Calculus %
%%%%%%%%%%%%%%%%%%%%%%%%

\section{Review of the Symbol Calculus} \label{SYMBOL CALCULUS}

Here we review Duistermaat and Guillemin's \cite{DG} symbol calculus for Fourier integral operators with cleanly composing canonical relations. This is also treated in H\"ormander's fourth book in his series on linear partial differential equations \cite{HormanderIV}. Afterwards, we describe how to compute symbols for transversally composing operators as a special case.

\subsection{For clean compositions}

Let $X$ and $Y$ be smooth manifolds, let $\mathcal C \subset \dot T^*X \times \dot T^*Y$ be a canonical relation and $\Lambda \subset \dot T^*Y$ a conic Lagrangian submanifold. To $\mathcal C$ we associate a Fourier integral operator $A$ and to $\Lambda$ we associate a Lagrangian distribution $B$. The goal is to determine that $A \circ B$ is a Lagrangian distribution on $Y$ associated to $\mathcal C \circ \Lambda$ and to compute its order and principal symbol. To do so, we require $\mathcal C$ and $\Lambda$ to compose cleanly. This is a condition on the fiber product
\begin{equation}\label{F fiber product}
	F = \{(a;b) \in \mathcal C : b \in \Lambda\}.
\end{equation}
of $\mathcal C$ and $\Lambda$.

\begin{definition}[Clean composition] \label{def clean composition} We say $\mathcal C$ and $\Lambda$ \emph{compose cleanly} if $F$ is a smooth manifold and
\[
	T_{(a;b)} F = \{(a';b') \in T_{(a;b)} \mathcal C : b' \in T_b \Lambda \}
\]
for each $(a;b) \in F$.
\end{definition}

The definition can be loosely read as, ``The tangent space of the fiber product is the fiber product of the tangent spaces." The composition 
\[
	\mathcal C \circ \Lambda = \{a \in \dot T^*X : (a;b) \in \mathcal C \text{ and } b \in \Lambda \}
\]
is the image of $F \to \dot T^*X$. The \emph{excess fiber}
\[
	E_a = \{b \in \Lambda : (a;b) \in \mathcal C\}
\]
over $a \in \mathcal C \circ \Lambda$ is the preimage of $a$ through $F \to \dot T^*X$. All of the following hold provided $\mathcal C$ and $\Lambda$ compose cleanly \cite{DG}.
\begin{enumerate}
	\item $\mathcal C \circ \Lambda$ is a (smooth) conic Lagrangian submanifold of $\dot T^*X$.
	\item For every $(a;b) \in F$,
	\[
		T_a (\mathcal C \circ \Lambda) = T_{(a;b)} \mathcal C \circ T_b \Lambda.
	\]
	\item $E_a$ is a smooth manifold for each $a \in \mathcal C \circ \Lambda$.
	\item If $Y$ is compact, then $E_a$ is compact for each $a$.
	\item For each $(a;b) \in F$,
	\[
		T_b E_a = \{b' \in T_b \Lambda : (0;b') \in T_{(a;b)} \mathcal C \}
	\]
	and
	\[
		\dim T_b E_a = (\dim T_{(a;b)} F + \dim T^*Y) - (\dim \mathcal C + \dim \Lambda).
	\]
\end{enumerate}
The last point implies that if the fiber $F$ has uniform dimension (as tends to be the case in practice), so do the excess fibers $E_a$. In this case we denote $\dim E_a$ by $e$ and call it the \emph{excess} of the composition. Given this setup, Duistermaat and Guillemin describe a canonical homogeneous half-density $\sigma_A \circ \sigma_B$ on $\mathcal C \circ \Lambda$, where $\sigma_A$ and $\sigma_B$ are the half-density parts of the principal symbols of $A$ and $B$, respectively, and show it determines the half-density part of the symbol of the composition $A \circ B$. We state the main theorem of their calculus \cite[Theorem 5.4]{DG} below for convenience, and then describe how they obtain $\sigma_A \circ \sigma_B$.

\begin{theorem}[Duistermaat and Guillemin] Let $\mathcal C$ and $\Lambda$ compose cleanly with excess $e$. If $A$ is a Fourier integral operator associated to $\mathcal C$ and $B$ a Lagrangian distribution associated to $\Lambda$, then $A \circ B$ is a Lagrangian distribution associated to $\mathcal C \circ \Lambda$ with order
\[
	\ord(A \circ B) = \ord A + \ord B + \frac{e}{2}.
\]
Furthermore, if $\sigma_A$, $\sigma_B$, and $\sigma_{A \circ B}$ are the half-density parts of the principal symbols of $A$, $B$, and $A \circ B$, respectively, then
\[
	\sigma_{A \circ B} = (2\pi i)^{-e/2} \sigma_A \circ \sigma_B.
\]
\end{theorem}

In what follows, we will review the canonical isomorphism
\begin{equation}\label{canonical isomorphism}
	|T_{(a;b)} \mathcal C|^{1/2} \otimes |T_b \Lambda|^{1/2} \simeq |T_b E_a| \otimes |T_a \mathcal C \circ \Lambda|^{1/2}
\end{equation}
as described in \cite{DG}. This will associate $\sigma_A \otimes \sigma_B$ in the space on the left with an object $\sigma_A \boxtimes \sigma_B$ in the space on the right. This object can be integrated over the excess fibers to obtain a half-density on $\mathcal C \circ \Lambda$. Indeed, we have
\begin{equation} \label{integral over excess fibers}
	\sigma_A \circ \sigma_B(a) = \int_{E_a} \sigma_A \boxtimes \sigma_B.
\end{equation}
Since we will only be considering $Y$ compact, $E_a$ is compact, and the integral is finite.

Now we describe how \eqref{canonical isomorphism} is obtained. We will be working exclusively in the linear category, so we can ease a burden on the notation by replacing some letters. We will write $\mathcal C$, $\Lambda$, $F$ and $\mathcal C \circ \Lambda$ to denote their respective tangent spaces, we will write $E$ to denote $T_b E_a$, and write $Y$ to denote the symplectic tangent space $T_b \dot T^* Y$.

Let $\beta : \mathcal C \times \Lambda \to Y$ be the map $\beta((a';b'),c') = b' - c'$. Then, we have an exact sequence
\begin{equation*}\label{long exact sequence}
	0 \longrightarrow F \longrightarrow \mathcal C \times \Lambda  \overset{\beta}{\longrightarrow} Y \longrightarrow \coker \beta \longrightarrow 0
\end{equation*}
which, as described in Appendix \ref{DENSITIES}, induces a linear isomorphism
\[
	|\mathcal C \times \Lambda|^{1/2} \simeq |F|^{1/2} \otimes |Y|^{1/2} \otimes |\coker \beta|^{-1/2}.
\]
Note, $Y$ is symplectic and so comes already equipped with the symplectic half-density. By fixing this element in $|Y|^{1/2}$, we have
\begin{equation}\label{clean calculus eq 1}
	|\mathcal C \times \Lambda|^{1/2} \simeq |F|^{1/2} \otimes |\coker \beta|^{-1/2}.
\end{equation}

As proved in \cite{DG}, $E$ is the symplectic orthogonal subspace to the image of $\beta$ in $Y$, and hence the symplectic form is well-defined and nondegenerate as a bilinear form on $\coker \beta \times E$. Also by Appendix \ref{DENSITIES}, this induces an isomorphism
\begin{equation}\label{clean calculus eq 2}
	|\coker \beta|^{-1/2} \simeq |E|^{1/2}.
\end{equation}
Thirdly, if $\alpha : F \to \mathcal C \circ \Lambda$ denotes projection onto the first factor, then we have another exact sequence
\[
	0 \longrightarrow E \longrightarrow F \overset{\alpha}{\longrightarrow} \mathcal C \circ \Lambda \longrightarrow 0,
\]
which induces another linear isomorphism
\begin{equation}\label{clean calculus eq 3}
	|F|^{1/2} \simeq |E|^{1/2} \otimes |\mathcal C \circ \Lambda|^{1/2}.
\end{equation}
Stringing together \eqref{clean calculus eq 1}, \eqref{clean calculus eq 2}, and \eqref{clean calculus eq 3}, we obtain
\[
	|\mathcal C \times \Lambda|^{1/2} \simeq |F|^{1/2} \otimes |\coker \beta|^{-1/2} \simeq (|E|^{1/2} \otimes |\mathcal C \circ \Lambda|^{1/2}) \otimes |E|^{1/2} \simeq |E| \otimes |\mathcal C \circ \Lambda|^{1/2}
\]
as desired.

\subsection{For transverse compositions} We review the symbol calculus for where $\mathcal C$ and $\Lambda$ compose transversally. This is simply a special case of clean composition when $e = 0$, but came first historically speaking. The calculus of transversally composing Fourier integral operators is treated in Duistermaat's book \cite{DuistermaatFIOs} and in H\"ormander's paper \cite{HormanderPaper} and its sequel with Duistermaat \cite{HormanderPaper2}. Again, a very thorough treatment of both clean and transversal cases can be found in \cite{HormanderIV}.

$\mathcal C$ and $\Lambda$ are said to \emph{compose transversally} if the intersection
\[
	F = \{(a;b) \in \mathcal C : b \in \Lambda \} = \mathcal C \cap (T^*X \times \Lambda)
\]
is transversal, meaning that
\[
	T_{(a;b)} \mathcal C + T_{(a;b)}(T^*X \times \Lambda) = T_{(a;b)} T^*(X \times Y).
\]
As a consequence, the intersection $F$ is automatically smooth and has tangent space
\[
	T_{(a;b)} F = T_{(a;b)} \mathcal C \cap (T_aT^*X \times T_b\Lambda) = \{(a';b') \in T_{(a;b)} \mathcal C : b' \in T_b \Lambda \},
\]
and hence the composition is clean. We also have
\begin{align*}
	\dim T_{(a;b)} F &= \dim \mathcal C + \dim (T^*X \times \Lambda) - \dim(T^*X \times T^*Y) \\
	&= \dim \mathcal C + \dim \Lambda - \dim T^*Y,
\end{align*}
and hence the excess $e$ of the composition is $0$. This significantly simplifies the composition calculus above.

First, the excess fibers are $0$-dimensional and compact, meaning each is a finite set. The densities $|T_b E_a|$ in \eqref{canonical isomorphism} are trivialized by the remark after Definition \ref{def alpha density}, and hence the object $\sigma_A \boxtimes \sigma_B$ belongs to $|T_a \mathcal C \circ \Lambda|^{1/2}$. The integral \eqref{integral over excess fibers} reduces to a finite sum. The procedure for obtaining $\sigma_A \boxtimes \sigma_B$ also simplifies. First, $\coker \beta$ is trivial, and hence the exact sequence involving $\beta$ simplifies to
\[
	0 \longrightarrow F \longrightarrow \mathcal C \times \Lambda \overset \beta \longrightarrow Y \longrightarrow 0.
\]
Again by fixing the symplectic half-density on $Y$, we have
\[
	|\mathcal C \times \Lambda|^{1/2} \simeq |F|^{1/2} \otimes |Y|^{1/2} \simeq |F|^{1/2}
\]
by the procedure in Appendix \ref{DENSITIES}. The exact sequence involving $\alpha$ then describes an isomorphism
\[
	0 \longrightarrow F \overset \alpha \longrightarrow \mathcal C \circ \Lambda \longrightarrow 0,
\]
which yields the identification
\[
	|F|^{1/2} \simeq |\mathcal C \circ \Lambda|^{1/2}.
\]
Putting these together yields a simplified \eqref{canonical isomorphism},
\begin{equation}\label{canonical isomorphism transversal}
	|\mathcal C \times \Lambda|^{1/2} \simeq |\mathcal C \circ \Lambda|^{1/2}.
\end{equation}

%%%%%%%%%%%%%%%%
% Bibliography %
%%%%%%%%%%%%%%%%

\bibliographystyle{plain}
\bibliography{references}

\end{document}